\documentclass[12pt]{amsart}
\usepackage{amsfonts}
\usepackage{amsmath,amssymb}
\usepackage{mathrsfs}
\usepackage{hyperref}
\usepackage{graphicx}
\usepackage{float}
\usepackage{amsfonts}
\usepackage{bbm}
\usepackage{amsmath,amscd,amsbsy,amssymb,latexsym,url,bm,amsthm}
\usepackage{mathrsfs, euscript}
\usepackage{lineno}
\usepackage{color}

\newtheorem{thm}{Theorem}[section]
\newtheorem{cor}[thm]{Corollary}
\newtheorem{lem}[thm]{Lemma}
\newtheorem{prop}[thm]{Proposition}
\newtheorem{defn}[thm]{Definition}

\newtheorem{rem}[thm]{Remark}

\makeatletter \renewenvironment{proof}[1][Proof]
{\par\pushQED{\qed}\normalfont\topsep6\p@\@plus6\p@\relax\trivlist\item[\hskip\labelsep\bfseries#1\@addpunct{.}]\ignorespaces}{\popQED\endtrivlist\@endpefalse} \makeatother
\makeatletter
\numberwithin{equation}{section}\allowdisplaybreaks

\def\leq{\leqslant}
\def\geq{\geqslant}
\def\les{\lesssim}
\def\no{\nonumber}

\def\R{{\mathbb{R}}}
\def\C{{\mathbb{C}}}

\def\FF{{\mathscr{F}^{-1}}}
\def\F{{\mathscr{F}}}
\def\Z{{\mathbb{Z}}}
\renewcommand{\S}{\mathcal{S}}

\newcommand{\pa}{\partial}

\newcommand{\norm}[1]{\left\|#1\right\|}

\title[Well-Posedness for KdV-Type Equations]
{Well-Posedness for KdV-Type Equations with Second-Order Derivative Nonlinearities}

\author[Well-Posedness for KdV-Type Equations]{\bf Mingjuan Chen$^{a}$,  Ting Chen$^{a}$, and Zhong Wang$^{b}$}

 \address{$^a$ School of Mathematics, Jinan University, Guangzhou 510632, P.R. China}
 \email{mjchen@jnu.edu.cn}

\address {$^b$ School of Mathematics, Foshan University, Foshan 528000, P. R. China}
\email{wangzh79@fosu.edu.cn}


\date{}

\begin{document}
	
	

	\maketitle

	\begin{abstract} 
  We study a class of complex-valued KdV-type equations with cubic
second-order derivative nonlinearities and arbitrary complex
coefficients. For sufficiently small initial data in $L^2$, we
prove local well-posedness in $H^s(\mathbb R)$ for $s\ge3/4$.
The key ingredient is a family of dyadic resolution spaces
$Z_k=X_k+Y_k$, designed to overcome the logarithmic divergence
arising from high$*$low$*$low interactions in which both
derivatives fall on the highest-frequency factor. 
For the completely integrable third-order Kaup--Newell flow, we
establish global-in-time $H^s(\mathbb R)$ bounds for $0\le s<1$
and global well-posedness in $H^s(\mathbb R)$ for every
$s\ge3/4$, with no smallness assumption on the initial data.

		
	\end{abstract}

	\section{Introduction}
The first part of this paper develops a low regularity well-posedness theory for complex-valued KdV-type equations whose nonlinearities contain two
spatial derivatives. More precisely, we consider the following initial value problem
	\begin{equation}\label{eq0}
		\begin{cases}
			 u_t -u_{xxx} = \mathcal N_3(u,\bar u) +\mathcal N_5(u,\bar u), \qquad (x,t)\in\R\times\R,\\
			u(x,0) = u_0(x),
		\end{cases}
	\end{equation}
where $\mathcal N_3$ is a cubic nonlinearity involving a total of two
spatial derivatives. To preserve the scaling symmetry of the equation,
and motivated by the nonlinear structures arising in the underlying
physical models, we take $\mathcal N_5$ to be a quintic nonlinearity
involving one spatial derivative. In both $\mathcal N_3$ and
$\mathcal N_5$, the coefficients may be arbitrary complex numbers, and
complex conjugates may be placed on any of the factors.

The linear group associated with Equation \eqref{eq0} is the Airy group $W(t)=e^{t\partial_x^3}$, and hence the corresponding dispersion relation is $\omega(\xi)=-\xi^3$. There is an extensive literature on the low regularity well-posedness theory for equations with this dispersive structure, most notably the KdV equation, with quadratic derivative nonlinearity $(u^2)_x$, and the modified KdV equation, with cubic derivative nonlinearity $(u^3)_x$. The Fourier restriction norm method of Bourgain \cite{Bourgain1993} and the multilinear estimates of Kenig, Ponce, and Vega \cite{KenigPonceVega1993,KenigPonceVega1996} established a robust framework for first-order derivative nonlinearities, see also \cite{CCT,A4}. More recently, exploiting the integrable structure through the method of commuting flows, Killip and Vi\c{s}an \cite{KV} established global well-posedness for the KdV equation in $H^{-1}(\mathbb R)$, pushing the regularity threshold well below the range $s\geq -3/4$ obtained by contraction mapping arguments. Subsequently, Harrop-Griffiths, Killip, and Vi\c{s}an \cite{HarropGriffithsKillipVisan2024} extended this approach to the modified KdV equation, establishing well-posedness in $H^s(\mathbb R)$ for any $s>-1/2$, far below the previously known threshold $s\geq 1/4$.

KdV-type equations with second-order derivative nonlinearities are substantially more delicate. General high-order results were obtained by Kenig, Ponce, and Vega \cite{KenigPonceVega1994} in the weighted Sobolev spaces $H^s(\R) \cap L^2(|x|^{k} dx)$ for sufficiently large $s>0$ and $k\in \mathbb{Z}^+$. Subsequent developments used weighted Sobolev spaces, local energy spaces, or refined frequency decompositions \cite{Pilod2008,HarropGriffiths2015,HirayamaKinoshitaOkamoto2020}.
For the model $u_t+u_{xxx}=\partial_x^2(u^{m+1})$ ($m\geq4$), global small data results in
critical modulation spaces were recently proved by Lu \cite{Lu2025}. In \cite{ChenChenXiao2026}, the first two authors of the present paper and Xiao considered cubic KdV-type equations in which the two spatial derivatives are distributed among different factors, a structure that can still be treated within
the classical Bourgain space framework. That work established local
well-posedness in $H^s(\R)$ for $s\geq 3/4$ and showed that this regularity threshold
is sharp, in the sense that the data-to-solution map fails to be
$C^3$ when $s<3/4$.  

Gauge transformations have played an important role in the study of the classical derivative nonlinear Schr\"odinger (dNLS) equation \cite{A13}, where they are used to eliminate or reorganize the most unfavorable derivative nonlinearities. More recently, this approach has been extended to
higher-order equations in the dNLS hierarchy to establish well-posedness in Fourier--Lebesgue and modulation spaces \cite{Adams2025}. In a related but structurally distinct nonlocal integrable model, G\'erard and Lenzmann \cite{GerardLenzmann2024} used a gauge transformation for the Calogero--Moser dNLS equation to eliminate the local derivative
component of the nonlinearity, leaving a nonlocal derivative interaction and revealing a standard Hamiltonian structure with a nonnegative energy. However, a fundamental limitation of this
approach is that the cancellation of the highest-order derivative term relies on a special algebraic structure of the nonlinearity. In particular, for a real-valued phase gauge, the
coefficient of the term to be eliminated must be purely imaginary. Consequently, such a gauge mechanism is not available for the general complex coefficients considered in this part. In the second part, however, the coefficients of the third-order
Kaup--Newell flow possess precisely the algebraic structure
required for the gauge transformation, and we shall exploit this structure in the corresponding analysis.

For Equation \eqref{eq0}, the principal difficulty arises from the term $u_{xx}|u|^2$, where both derivatives are concentrated on the highest frequency factor. This derivative concentration leads to a
delicate high$*$low$*$low interaction, for which a direct iteration in the standard Bourgain space $X^{s,b}$ incurs a logarithmic divergence. To compensate for this loss, we introduce high frequency resolution spaces of the form $Z_k=X_k+Y_k$, where $X_k$ is a weighted, dyadically localized Fourier restriction space, while $Y_k$ is designed to capture the local smoothing effect through an $L_x^1L_t^2$ norm. To the best of our knowledge, this is the first time that such a local smoothing component has been incorporated into dyadically localized resolution spaces for an Airy-type  equation. Closely related mechanisms arise in the low regularity theory for Benjamin--Ono (BO), modified BO equations and BO-Burgers equations (see \cite{IonescuKenig07,Guo2011,CGH22}). Related high$*$low frequency difficulties for fourth-order derivative Schr\"odinger equations were studied in \cite{HuoJia2011}. We also note that dyadically localized $X_k$ spaces have previously been used in the Airy setting. In particular, in his proof of global well-posedness for the KdV equation in $H^{-3/4}(\mathbb R)$, Guo \cite{Guo2009} employed unweighted dyadically localized $X_k$ spaces to handle high$*$high frequency
interactions.

A further ingredient in our argument is Tao's $[k;Z]$ multiplier formalism \cite{Tao2001}, which we adapt to the Airy resonance structure arising in the present problem (see Lemma~\ref{4Zmultiplier}). In particular, this formulation allows us to organize the dyadic multilinear estimates according to the Airy resonance geometry. 

Our first main result is the following.

\begin{thm}\label{thm-lwp}
Let $s\geq 3/4$. Assume that $u_0\in H^s(\R)$ and $\norm{u_0}_{L^2}\leq\delta$ for some sufficiently small $\delta>0$. Then there exists
$T=T(\norm{u_0}_{H^s})>0$ and a unique local solution to \eqref{eq0} satisfying
$$
 u\in W^s(T)\hookrightarrow C([-T,T];H^s(\R)).
$$
Moreover, for any $R>0$, the data-to-solution map is locally Lipschitz continuous from 
$\{u_0\in H^s(\mathbb{R}):\|u_0\|_{H^s}\leq R,\
\|u_0\|_{L^2}\leq\delta\}
$
into $C([-T,T];H^s(\mathbb{R}))$.
\end{thm}
\begin{rem} Equation \eqref{eq0} is invariant under the scaling
\begin{equation*}
u_\lambda(x,t)
 :=\lambda^{1/2}u(\lambda x,\lambda^3t).
\end{equation*}
In particular,
$\norm{u_\lambda(x,0)}_{L^2}=\norm{u_0}_{L^2}$, so $L^2(\R)$ is the
scaling-critical Sobolev space. Thus, the smallness assumption on the $L^2$ norm in our results cannot
be removed by scaling.
\end{rem}

Equation \eqref{eq0} includes several equations of physical and
mathematical interest. One notable example is the Calogero--Degasperis--Ibragimov--Shabat (CDIS) equation
\begin{equation*}
		u_t = u_{xxx} + 3u_{xx} u^2 + 9 u_x^2 u + 3 u^4u_x. 
\end{equation*}
The CDIS equation can be exactly linearized by a suitable transformation of the dependent variable and admits a variety of explicit solutions (see \cite{FC87,EE}). Remarkably, some of these solutions exhibit soliton-like behavior. Another important example is the third-order flow in the Kaup--Newell
hierarchy \cite{KaupNewell1978}, which we study as a distinguished
completely integrable model in the second part of this paper:
\begin{equation}\label{maineq}
		\begin{cases}
			 u_t
 = u_{xxx}
    +3i\big(|u|^2u_x\big)_x
    -\frac32\,\bigl(|u|^4u\bigr)_x,\\
			u(x,0) = u_0(x),
		\end{cases}
	\end{equation}
where $u$ is a complex valued function of $(x,t)\in \mathbb{R} \times  \mathbb{R}$. It combines Airy dispersion with higher-order derivative nonlinearities while retaining the integrable structure inherited from the Kaup--Newell hierarchy. This structure will play an essential role in the construction of the low regularity conservation laws. Equivalent versions occur in the literature with different signs and coefficients. These versions \cite{GengLiXueGuan2015,ZhuChen2021} are related by time reversal, constant rescaling of time, and complex conjugation. This third-order flow admits a Lax representation, an infinite number of conservation laws, and a
multi-Hamiltonian structure, see \cite{Fan2001,GengLiXueGuan2015,ZhuChen2021} and the references therein. The standard dNLS equation, which arises in models of nonlinear wave propagation in nonlinear optics and plasma physics, is the second-order flow in the same Kaup--Newell hierarchy.

The standard Kaup--Newell spectral problem is
\begin{equation}
\label{standard-spatial}
 \Psi_x=U(\lambda;q,r)\Psi,
 \qquad
 U(\lambda;q,r)=-i\lambda^2\sigma_3+\lambda Q,
\end{equation}
where
\begin{equation*}
 \sigma_3=\begin{pmatrix}1&0\\0&-1\end{pmatrix},
 \qquad
 Q=\begin{pmatrix}0&q\\r&0\end{pmatrix}.
\end{equation*}
This is the common spatial operator for all flows in the Kaup--Newell
hierarchy, whereas the individual flows are distinguished by their time matrices.
Its Hamiltonian structure (see \cite{KaupNewell1978,Sasaki1982,Adams2025}) is 
\begin{equation}
\label{hamiltonian-flow}
 \pa_t\binom{r}{q}
 =\mathcal{J}\frac{\delta}{\delta(r,q)}
 \left(\sum_{n\geq0}\alpha_n I_n\right), 
\end{equation}
where $$\mathcal{J} :=-2i
 \begin{pmatrix}0&1\\1&0\end{pmatrix}
 \pa_x,$$
 and
\begin{equation*}
 I_n=\int_{\R}qY_n\,dx,
 \qquad
 Y_0=-\frac{r}{2i},
 \qquad
 Y_{n+1}=\frac{1}{2i}
 \left(\pa_xY_n+q\sum_{k=0}^{n}Y_{n-k}Y_k\right). 
\end{equation*}
The dNLS equation ($n=1$)
\begin{align*}
iu_t+u_{xx}=i(|u|^2u)_x
\end{align*}
and its higher-order hierarchy are obtained by retaining the odd-indexed
coefficients $\alpha_{2j-1}$ in \eqref{hamiltonian-flow}, which are of even dispersion order $2j$.  The even-indexed coefficients $\alpha_{2j}$ generate a sequence of flows with odd dispersion order $2j+1$.  Equation \eqref{maineq} is the first nontrivial member of this sequence, corresponding to $n=2$.

Keeping only $\alpha_2$ in \eqref{hamiltonian-flow} gives the coupled flow
\begin{equation}
\label{coupled-n2}
\left\{
\begin{aligned}
 q_t&=-\frac{\alpha_2}{4}
 \left\{q_{xxx}
 +\pa_x\left(-3iqq_xr-\frac32q^3r^2\right)\right\},\\
 r_t&=-\frac{\alpha_2}{4}
 \left\{r_{xxx}
 +\pa_x\left(3iqrr_x-\frac32q^2r^3\right)\right\}.
\end{aligned}
\right.
\end{equation}
For real $\alpha_2$, the reduction $r=-\overline q$ is preserved by \eqref{coupled-n2}.  With $\alpha_2=-4$ and $q=u$, the first equation in \eqref{coupled-n2} becomes exactly \eqref{maineq}.  Set
$$
 \Omega_+:=\{\lambda\in\C:{\rm Im}(\lambda^2)>0\}.
$$
For $u\in\S(\R)$ and $\lambda\in\Omega_+$, let $\psi_1^-(x,\lambda)$ and
$\psi_2^+(x,\lambda)$ denote the Jost solutions characterized by the asymptotic conditions
\begin{align*}
\psi_1^-(x,\lambda)
&=
e^{-i\lambda^2x}
\left[
\begin{pmatrix}
1\\
0
\end{pmatrix}
+o(1)
\right],
&& x\to-\infty,\\
\psi_2^+(x,\lambda)
&=
e^{i\lambda^2x}
\left[
\begin{pmatrix}
0\\
1
\end{pmatrix}
+o(1)
\right],
&& x\to+\infty.
\end{align*}

We define the corresponding Wronskian by
\begin{equation*}
 a(\lambda;u)
 :=\det\bigl(\psi_1^-(x,\lambda),\psi_2^+(x,\lambda)\bigr).
\end{equation*}
The Riccati equation yields that $\log a(\lambda;u)$ admits the following asymptotic expansion
\begin{equation}\label{log-a}
\log a(\lambda;u)
= \sum_{j=0}^{\infty}\frac{E_j(u)}{\lambda^{2j}},
\qquad |\lambda|\to\infty,\quad \lambda\in\overline{\Omega}_+,
\end{equation}
where the quantities $E_j(u)$ are conserved along the flow. The first
two are given by
\begin{align*}
E_0(u)
&=-\frac{i}{2}\int_{\mathbb R}|u|^2\,dx,\\
E_1(u)
&=-\frac14\int_{\mathbb R}u\overline{u}_x\,dx
+\frac{i}{8}\int_{\mathbb R}|u|^4\,dx.
\end{align*}
Moreover, $a(\lambda;u)$ can be realized as a perturbation determinant. We convert the preceding Lax pair \eqref{standard-spatial} into the operator normalization used in the low regularity perturbation determinant. Let
\begin{equation*}
 \kappa=-i\lambda^2,
 \qquad
 G_\lambda={\rm diag}(1,i\lambda),
 \qquad
 \Phi=G_\lambda^{-1}\Psi.
\end{equation*}
For $\kappa>0$ one may take $\lambda=e^{i\pi/4}\sqrt{\kappa}$. Under the reduction
$q=u$, $r=-\overline u$, the spatial equation becomes
\begin{equation*}
 \Phi_x=\mathcal U_{\kappa}(u)\Phi,
 \qquad
 \mathcal U_{\kappa}(u)=
 \begin{pmatrix}
 \kappa&-\kappa u\\
 i\overline u&-\kappa
 \end{pmatrix}.
\end{equation*}
Thus the corresponding first-order operator is
\begin{equation*}
 L_{\kappa}(u):=\pa_xI-\mathcal U_{\kappa}(u)
 =\begin{pmatrix}
 \pa_x-\kappa&\kappa u\\
 -i\overline u&\pa_x+\kappa
 \end{pmatrix}.
\end{equation*}
By computing the logarithmic determinant
\begin{align*}
	-\log\det\left( \begin{bmatrix}(\partial_x-\kappa)^{-1}&0\\0&(\partial_x+\kappa)^{-1}\end{bmatrix}
	\begin{bmatrix}\partial_x-\kappa & \kappa u \\-i\bar{u} & \partial_x+\kappa\end{bmatrix} \right),
\end{align*}
and performing a simple Taylor expansion, we obtain a conserved series of the form
\begin{align}\label{alpha}
	\alpha(\kappa;u)=-\log a(\sqrt{i\kappa};u)=\sum_{m=1}^{\infty}\frac{(-i\kappa)^m}{m}{\rm tr}\Bigl\{[(\partial_x-\kappa)^{-1}u(\partial_x+\kappa)^{-1}\bar{u}]^m\Bigr\}.
\end{align}

The construction of low regularity conservation laws through perturbation determinants was introduced by Killip, Vi\c{s}an,
and Zhang \cite{KillipVisanZhang2018} and subsequently applied
to the dNLS equation by Klaus and Schippa \cite{KlausSchippa2022} and Bahouri, Leslie, and Perelman
\cite{BLP2023}. We adapt this approach to the third-order Kaup--Newell flow. 

For sufficiently large $\kappa$, the series  in \eqref{alpha}  converges and may be differentiated term by term along a smooth solution. The Lax equation and the cyclicity of the trace imply that
$\frac{d}{dt}\alpha(\kappa;u(t))=0$. Equivalently, the time independence of $a(\lambda;u(t))$ follows from the time component of the Lax pair.  The imaginary part of the quadratic term in \eqref{alpha} gives a weighted $L^2$ quantity,
from which Sobolev norms can be recovered through suitable frequency weights. The main task is therefore to control the higher-order terms.

To obtain bounds for arbitrary initial mass, we use the large-data equicontinuity of spectral level sets established by Harrop-Griffiths, Killip, and Vi\c{s}an \cite{HGKV2023}. Refined estimates for the higher-order terms and the equicontinuity in $L^2$ space then yield a priori bounds without a smallness assumption on the initial mass. We now state the second main result.

\begin{thm}\label{thm-apriori}
Let $0\leq s<1$. If $u$ is the solution to \eqref{maineq} with initial data $u_0\in \S(\R)$. Then there exists a function $C_s:[0,\infty)\rightarrow [0,\infty) $ so that 
\begin{equation}\label{apriori0}
 \sup_{t}\norm{u(t)}_{H^s(\R)}
 \leq C_s(\norm{u_0}_{H^s(\R)}).
\end{equation}
\end{thm}

For Equation \eqref{maineq}, the second-order derivative nonlinearity has a purely imaginary coefficient. By means of a gauge transformation, we can eliminate the worst term $|u|^2u_{xx}$. Specifically, we set
\begin{align*}
    v(x,t) = \mathcal{G}(u)(x,t) = e^{i \int_{-\infty}^{x} |u(y,t)|^2 dy} u(x,t),
\end{align*}
then Equation \eqref{maineq} can be rewritten as 
\begin{equation}\label{maineq2}
    \begin{cases}
         v_t=v_{xxx}-3iv|v_x|^2+\frac32|v|^4v_x,\\[2mm]
        v(x,0)= e^{i\int_{-\infty}^x|u_0(y)|^2\,dy} u_0(x).
    \end{cases}
\end{equation}
Since the map $\mathcal{G}$ is a diffeomorphism of $H^s(\R)$ into itself for every $s\geq 0$, the well-posedness of \eqref{maineq} in $H^s$ is equivalent to that of \eqref{maineq2}. In \cite{ChenChenXiao2026}, the authors proved that \eqref{maineq2} is locally well-posed in  $H^s$ for $s\geq 3/4$ by the standard Bourgain method. 
Combining this local theory with the a priori estimates \eqref{apriori0} yields global well-posedness for $3/4\le s<1$. Persistence of higher regularity then extends this conclusion to all $s\ge3/4$. We therefore obtain the following result.
\begin{cor}\label{globalwell}
 Let  $s\geq 3/4$ and $u_0 \in  H^{s}(\R)$. Then Equation \eqref{maineq} is globally well-posed.   
\end{cor}
The paper is organized as follows. In Section~2 we introduce the function spaces $X_k$, $Y_k$, $Z_k$, $W^s$ and $V^s$, establish the
linear Airy estimates, and prove the crucial localized trilinear estimates.  In Section~3, we establish the trilinear and quintilinear estimates and prove Theorem \ref{thm-lwp}. In Section~4, we prove the 
conservation of the Kaup--Newell perturbation determinant for
\eqref{maineq}. Using equicontinuity together with precise estimates
for the remainder terms, we first derive a priori bounds for
$0\leq s<1/2$ and then extend these bounds to the higher regularity
range $1/2\leq s<1$, thereby completing the proof of
Theorem~\ref{thm-apriori}. 

{\bf Notations.} Let $c < 1$, $C>1$  denote positive universal constants, whose values may vary from line to line. We use the standard asymptotic notation $x\lesssim y$ to mean $x\leq Cy$ for some constant $C>0$, and write $a\sim b$ if $a\lesssim b$ and $b\lesssim a$.  $a\approx b$ means that $|a-b|\leq C$. The notation $A \ll B$ indicates that there exists a small positive constant $c>0$ such that $A\leq cB$.  $\S(\R)$ denotes the Schwartz space and $\S'(\R)$ the space of tempered distributions. The notation $\mathcal{F}$ ($ \mathcal{F}^{-1}$) is used to denote the (inverse) Fourier transform operators on $ \mathcal{S}'(\mathbb{R} \times \mathbb{R}) $. For convenience, $\mathcal{F}_x$ ($ \mathcal{F}^{-1}_\xi $) and $\mathcal{F}_t $ ($ \mathcal{F}^{-1}_\tau $) are employed to represent the (inverse) Fourier transforms with respect to the spatial and temporal variables, respectively. We also write $\widehat{\phi}$ for the Fourier transform of $\phi$. Let $P_N$ be an inhomogeneous Littlewood--Paley decomposition, with $P_1$ containing
the low frequencies, and set $u_N=P_Nu$, $u_{>L}=\sum_{N>L}u_N$.
	
	\section{Preliminaries}
    In this section, we develop the functional framework for the
local well-posedness analysis. We begin by introducing the resolution
spaces and establishing their basic properties, followed by the linear
estimates for the Airy evolution. We then derive the $[4;Z]$ multiplier
estimates adapted to the Airy resonance structure of the cubic
nonlinearity. These estimates provide the main analytical ingredients
for the nonlinear estimates and the proof of local well-posedness in
Section 3.
	
	\subsection{Function spaces}
Denote $\mathbb{Z}_+=\mathbb{Z}\cap[0,\infty)$. Let $\eta_0:\R\to [0,1]$ be an even smooth non-negative function supported in $[-8/5, 8/5]$ and equal to 1 in $[-5/4, 5/4]$.
For $\ell\in\Z$, let the homogeneous dyadic decomposition functions $	\chi_{\ell }(\xi):=\eta_0(\xi/2^{\ell})-\eta_0(\xi/2^{\ell-1})$ and then ${\rm supp}\ \chi_{\ell}\subset \{\xi:|\xi|\in[(5/8)\cdot2^{\ell},(8/5)\cdot2^{\ell}]\}$. For brevity, let
		\begin{align*}
			\chi_{[\ell_1,\ell_2]}:=\sum_{\ell=\ell_1}^{\ell_2}\chi_{\ell},\quad \forall \ell_1\leq \ell_2\in\Z.
		\end{align*}
For the nonhomogeneous dyadic decomposition functions, let
		\begin{equation}
			\eta_{\ell}:=\left\{\begin{array}{l}
				\chi_{\ell},\quad\ell\geq 1,\\
				0,\quad\ell\leq -1,\nonumber
			\end{array}\right.
			\quad 	\eta_{[{\ell}_1,{\ell}_2]}:=\sum_{\ell=\ell_1}^{\ell_2}\eta_{\ell},\quad \eta_{\leq \ell_2}:=\sum_{\ell=-\infty}^{\ell_2}\eta_{\ell}.
		\end{equation}
We also define the homogeneous dyadic intervals
		\begin{align*}
			I_\ell =\{\xi\in\R:|\xi|\in[2^{\ell -1},2^{\ell +1}]\},\  \forall \ell \in \Z,
		\end{align*}
and the nonhomogeneous dyadic intervals, $\forall \ell \in \Z_+$, let	$\tilde{I_0}=[-2,2]$, and $\tilde{I_{\ell}}=I_{\ell}$ if $\ell\geq 1$.
Let 
$$\omega(\xi)=-\xi^3,$$
for $k\in\Z$ and $j\in\Z_+$, we define
\begin{equation*}
\begin{cases}
			&D_{k,j}=\{(\xi,\tau)\in \R \times \R:\xi\in{I_k},\tau-\omega(\xi)\in\tilde{I_{j}}\},\quad {\rm if} \ k\geq 1,\\
			&D_{k,j}=\{(\xi,\tau)\in \R \times \R:\xi\in{I_k},\tau\in\tilde{I}_{j}\},\qquad\qquad  {\rm if} \ k\leq 0.
\end{cases}
\end{equation*}

	\begin{defn}
		For each $k\in\Z_+$, we define the Besov-type Banach space $X_k=X_k(\R\times\R)$ as the collection of all functions $f\in L^2(\R\times\R)$ satisfying:
		\begin{align*}
			X_k	=\{&f\in L^2(\R\times\R):{\rm supp}\ f\subset \tilde{I_k} \times \R \quad and \no\\&\|f \|_{X_k}:=\sum_{j=0}^{\infty}2^{\frac{j}{2}}\big(1+2^{\frac{j-3k}{2}}\big)\|\eta_j(\tau-\omega(\xi))f(\xi,\tau) \|_{L_{\xi,\tau}^2}<\infty\}.
		\end{align*}
	\end{defn}
	
To control the logarithmic divergence phenomena arising from the nonlinear interaction term, we introduce auxiliary spatial structures in our functional framework.
	
	\begin{defn}\label{def1}
		For each $k \geq 100$, we define the \textit{smoothing effect space} $Y_k = Y_k(\R\times\R)$
		\begin{align*}
			Y_k:=\{&f\in L^2(\R\times\R):{\rm supp}\ f\subset \bigcup_{j=0}^{2k-1}D_{k,j} \quad and\no\\& \|f \|_{Y_k}=2^{-k}\|\FF[(\tau-\omega(\xi)+i)f(\xi,\tau)] \|_{L_x^1L_t^2}<\infty\}.
		\end{align*}
\end{defn}
\begin{defn}\label{defnW}
		For $k \in \mathbb{Z}_+$, we define
		\begin{align*}
			Z_k:=X_k \ \text{if}\  0\leq k\leq 99,\  \  \text{and}\   Z_k:=X_k+Y_k \ \text{if}\  k\geq100.
		\end{align*}
For $s \in \mathbb{R}$, we define the Banach spaces $W^s = W^s(\mathbb{R} \times \mathbb{R})$ and $V^s=V^s(\R\times \R)$ by
		\begin{align*}
			W^s&=\{u\in \S'(\R\times\R):\|u \|_{W^s}^2=\sum_{k=0}^{\infty}2^{2sk}\|\eta_k(\xi)\F(u) \|_{Z_k}^2<\infty\},\\
			V^s&=\{u\in\S'(\R\times \R):\|u \|_{V^s}^2=\sum_{k=0}^{\infty}2^{2sk}\|\eta_k(\xi)(\tau-\omega(\xi)+i)^{-1}\F(u)\|_{Z_k}^2<\infty\}.
		\end{align*}
	\end{defn}

	\subsection{Basic Estimates}
	
	For any integer  $k\in\Z_+$ and function $f_k\in Z_k$, we decompose $f_k$ as
	\begin{equation*}
		\left\{\begin{array}{l}	f_k=\sum_{j=0}^{\infty}f_{k,j}+g_k,\\\sum_{j=0}^{\infty}2^{j/2}\|f_{k,j}\|_{L^2}+\|g_k\|_{Y_k}\leq 2\|f_k\|_{Z_k},
		\end{array}\right.
	\end{equation*}
	where each $f_{k,j}$ has frequency support in $D_{k,j}$  and $g_k$ is supported in $\bigcup_{j=0}^{2k-1} D_{k,j}$(with $g_k\equiv0$ for $k\leq 99$). We now summarize the essential properties of the $Z_k$ space.
	\begin{lem}\label{aaa}
		\item[\rm (a)]If $h$, $m:\R\to \mathbb{C},k\in \Z_+$, and $f_k\in Z_k$, then
		\begin{align*}
			&\|h(\xi)f_k(\xi,\tau)\|_{Z_k}\les\|\F^{-1}_{\xi}(h)\|_{L^1(\R)}\|f_k\|_{Z_k}, \\
			&\|m(\tau)f_k(\xi,\tau)\|_{Z_k}\les\|m\|_{L^{\infty}(\R)}\|f_k\|_{Z_k}.
		\end{align*}
		\item[\rm (b)]If $k\in \Z_+$, $j\geq 0$, and $f_k\in Z_k$, then
		\begin{align}\label{b}
			\|\eta_j(\tau-\omega(\xi))f_k(\xi,\tau)\|_{X_k}\les \|f_k\|_{Z_k}.
		\end{align}
		\item[\rm (c)]If $k\geq 1$, $j\in [0,2k]$, and $f_k$ is supported in $I_k\times \R$, then
		\begin{align}\label{c}
			\|\FF[\eta_{\leq j}(\tau-\omega(\xi))f_k(\xi,\tau)]\|_{L_x^1L_t^2}\les \|\FF(f_k)\|_{L_x^1L_t^2}.	
		\end{align}
	\end{lem}
	
	\begin{proof}
		For (a), the proof in $X_k$ is obvious. Therefore, we consider $f_k \in Y_k$. From the definition of $Y_k$ and Young's inequality, we have
		\begin{align*}
			\|h(\xi)f_k(\xi,\tau)\|_{Y_k}\les2^{-k} \| \F^{-1}_{\xi} h\|_{L_{x}^{1}} \|\FF[(\tau-\omega(\xi)+i)f_{k}]\|_{L_x^1L_t^2}=\|\F^{-1}_{\xi}h\|_{L_{x}^{1}}\|f_{k}\|_{Y_k}.
		\end{align*}
Plancherel's identity implies that
		\begin{align*}
			\|m(\tau)f_k(\xi,\tau)\|_{Y_k}\les2^{-k} \| m(\tau)\|_{L_{\tau}^{\infty}}\|\F^{-1}_{\xi}[(\tau-\omega(\xi)+i)f_{k}(\xi,\tau)]\|_{L_x^1L_{\tau}^2}=\|m\|_{L_{\tau}^{\infty}}\|f_{k}\|_{Y_k}.
		\end{align*}
		For (b), it suffices to consider $k\geq 100$, $f_k=g_k\in Y_k$, and $j\leq 2k$. Then $g_k$ can be written as
		\begin{equation}\label{g_k}
			\left\{\begin{array}{l}	 g_k(\xi,\tau)=2^k\chi_{[k-1,k+1]}(\xi)(\tau-\omega(\xi)+i)^{-1}\eta_{\leq 2k}(\tau-\omega(\xi))\int_{\R}e^{-ix\xi}h(x,\tau)dx,\\	\|g_k\|_{Y_k}=C\|h\|_{L_x^1L_{\tau}^2}.
			\end{array}\right.
		\end{equation}
		Notice that $|\{\xi\in I_k:|\tau-\omega(\xi)|\leq 2^{j+1}\}|\les 2^{j-2k}$, by the definition of $X_k$ and H\"older's inequality, we obtain
		\begin{align*}
			\text{LHS}\eqref{b}&\les2^{j/2}2^k\Big \|\chi_{[k-1,k+1]}(\xi)(\tau-\omega(\xi)+i)^{-1}\eta_j(\tau-\omega(\xi))\int_{\R}e^{-ix\xi}h(x,\tau)dx\Big\|_{L_{\xi,\tau}^2}\\
&\les 2^{-j/2}2^k2^{(j-2k)/2}\|h\|_{L_x^1L_{\tau}^2}\les\|g_k\|_{Y_k}.
		\end{align*}
For (c), By using the Plancherel's theorem, the Hausdorff-Young inequality and H\"older's inequality,
		\begin{align*}
			\text{LHS}\eqref{c}=\,&\|\F_{\xi}^{-1}[\chi_{[k-1,k+1]}(\xi)\eta_{\leq j}(\tau-\omega(\xi))f_k(\xi,\tau)]\|_{L_x^1L_{\tau}^2}
			\\\les &\|\F_{\xi}^{-1}[\chi_{[k-1,k+1]}(\xi)\eta_{\leq j}(\tau-\omega(\xi))]\|_{L_x^1L_{\tau}^{\infty}}\|\F_{\xi}^{-1}(f_k)\|_{L_x^1L_\tau^2}.
		\end{align*}
		It suffices to prove that
		\begin{align}\label{shang}
			\Big\|\int_{\R}e^{ix\xi }\chi_{[k-1,k+1]}(\xi)\eta_{\leq j}(\tau-\omega(\xi))d\xi \Big \|_{L_x^1L_{\tau}^{\infty}}\les 1.
		\end{align}
		For $k\les 1$ and $|x|\leq 1$, we have LHS\eqref{shang} $\les 2^{k}\les 1$. For $k\les 1$ and $|x|\geq 1$, by simple calculations, we can obtain the following result:
		\begin{align}\label{jifen2}
			\Big |\int_{\R}e^{ix\xi}\chi_{[k-1,k+1]}(\xi)\eta_{\leq j}(\tau-\omega(\xi))d\xi\Big |=\,&\Big|\int\frac{-d_\xi^2(e^{ix\xi})}{x^2}\chi_{[-1,1]}(2^{-k}\xi)\eta_0(2^{-j}(\tau-\omega(\xi)))d\xi\Big|\no \\\les& \frac{1}{x^2}(2^{2k-j})^22^{j-2k}\les \frac{2^{2k-j}}{x^2}\les \frac{1}{x^2}.
		\end{align}
		The integral of the RHS\eqref{jifen2} with respect to $x$ is convergent. For $k\gg 1$, we can compute to obtain the following result:
		\begin{align}\label{jifen}
			\text{LHS}\eqref{jifen2}=\,&\Big|\int\frac{(I-2^{2(j-2k)}\Delta)(e^{ix\xi})}{1+(2^{j-2k}x)^2}\chi_{[k-1,k+1]}(\xi)\eta_{\leq j }(\tau-\omega(\xi))d\xi\Big|\no \\\les&\frac{1}{1+(2^{j-2k}x)^2}\cdot[1+(2^{2(j-2k)}\cdot2^{2(2k-j)})]\cdot2^{j-2k}\les\frac{2^{j-2k}}{1+(2^{j-2k}x)^2} .
		\end{align}
		Since the integral of the RHS\eqref{jifen}  with respect to $x$ converges, the proof is complete.			
	\end{proof}
	\begin{lem}[Kenig et al. \cite{K2,B10,K3}]\label{estimate}
		Let $I\subset \R$ be a interval with $|I|\les 1$ and  $k \in \Z$. Then for all  $\varphi\in L^2(\R)$ with $\hat{\varphi}$ supported in ${I_k}$, we have\\
(1) Strichartz estimates:	
\begin{align*}
&\|e^{t\partial_x^3}\varphi\|_{L_t^8L_x^8}\les \|\varphi\|_{L^2};\\
&\|e^{t\partial_x^3}\varphi\|_{L_x^6L_t^6}\les 2^{-\frac{k}{6}}\|\varphi\|_{L^2}.
\end{align*}
(2) Smoothing effect estimates:
\begin{align*}
\|e^{t\partial_x^3}\varphi\|_{L_x^{\infty}L_t^2}\les 2^{-k}\|\varphi\|_{L^2}.
\end{align*}
(3) Maximal function estimates:
\begin{align*}		
		&\|e^{t\partial_x^3}\varphi\|_{L_x^2L_{t\in I}^{\infty}}\les 2^{\frac{3k}{4}}\|\varphi\|_{L^2};\\	
&\|e^{t\partial_x^3}\varphi\|_{L_x^4L_t^{\infty}}\les2^{\frac{k}{4}}\|\varphi\|_{L^2}.
	\end{align*}
	\end{lem}
	\begin{lem}\label{lemma2.4}
		Let $k\in\Z_+$, $s\in \R$ and $I\subset\R$ be an interval. Let $Y$ be $L_{t\in I}^qL_x^p$ or $L_x^pL_{t\in I}^q$ for some $1\leq p,q\leq \infty$, and satisfy
		\begin{align*}
			\|e^{t\partial_x^3}f\|_Y\les2^{sk}\|f\|_{L^2(\R)},
		\end{align*}
		for all $f\in L^2(\R)$ with $\hat{f}$ supported in $\tilde{I_k}$. Then we have that if $f_k\in Z_k$
		\begin{align*}
			\|\FF(f_k)\|_Y\les2^{sk}\|f_k\|_{Z_k}.
		\end{align*}
	\end{lem}
	
	\begin{proof}
		Firstly, we assume that $f_k = f_{k,j}$ with $\|f_k\|_{X_k}=2^{j/2}(1+2^{(j-3k)/2})\|f_{k,j}\|_{L_{\xi,\tau}^2}$ and that $f_{k,j}$ is supported in $D_{k,j}$ for some $j\geq 0$. Then we obtain
		\begin{align*}
			\FF(f_k)(x,t)=\int e^{ix\xi+it\tau}f_{k,j}(\xi,\tau)d\xi d\tau
			=\int_{\tilde{I_j}} e^{it\tau}\int e^{ix\xi-it\xi^3}f_{k,j}(\xi,\tau-\xi^3) d\xi d\tau.
		\end{align*}
		An immediate consequence of the hypothesis on $Y$ and Minkowski inequality is
		\begin{align}\label{Xk}
			\|\FF(f_k)(x,t)\|_Y\les&\int \eta_j(\tau) \Big \|e^{it\tau}\int f_{k,j}(\xi,\tau-\xi^3)e^{ix\xi-it\xi^3} d\xi\Big \|_{Y} d\tau\nonumber\\
\les&2^{sk}2^{j/2}(1+2^{(j-3k)/2})\|f_{k,j}\|_{L_{\xi,\tau}^2}=2^{sk}\|f_k\|_{X_k}.
		\end{align}
		Secondly, we assume that $k\geq 100$ and $f_k = g_k \in Y_k$. According to the definition of $ Y_k $, we can express $g_k$ as \eqref{g_k}
$$g_k(\xi,\tau)=2^k\chi_{[k-1,k+1]}(\xi)(\tau-\omega(\xi)+i)^{-1}\eta_{\leq 2k}(\tau-\omega(\xi))\int_{\R}e^{-iy\xi}h(y,\tau)dy.$$
From the Minkowski inequality, fix $y$ and define 
$$f(\xi,\tau)=2^k\chi_{[k-1,k+1]}(\xi)(\tau+\xi^3+i)^{-1}\eta_{\leq 2k}(\tau+\xi^3)h(\tau),$$
it suffices to prove that
		\begin{align}\label{Ff}
			\big\|\FF \big(f(\xi,\tau)\big)\big\|_{Y}\les 2^{ks}\|h\|_{L_{\tau}^2}.
		\end{align}
		Since $|\xi|\in [2^{k-2},2^{k+2}]$, $|\tau+\xi^3|\leq 2^{2k+1}$, then $|\tau|\in [2^{3k-9},2^{3k+9}]$. By symmetry, we may assume $\tau\in[2^{3k-9},2^{3k+9}]$. Since $|\tau+\xi^3|\leq 2^{2k+1}$ and $\tau+\xi^3=(\tau^{1/3}+\xi)(\tau^{2/3}-\tau^{1/3}\xi+\xi^2)$, we have $\xi\in [-2^{k+2},-2^{k-2}]$ and $|\tau^{1/3}+\xi|\leq C$. Therefore, let
		\begin{align*}
			f'(\xi,\tau)=2^k\chi_{[k-1,k+1]}(-\tau^{1/3})\eta_0(\tau^{1/3}+\xi)(\tau+\xi^3-(\tau^{1/3}+\xi)^3-i2^{-2k}\tau^{1/3}\xi)^{-1}h(\tau),
		\end{align*}
and $\mu=|\tau+\xi^3|+1$. Through direct calculation, we obtain 
\begin{align*}
&\bigg|\chi_{[k-1,k+1]}(\xi)\eta_{\leq 2k}(\tau+\xi^3)\bigg(\frac{1}{\tau+\xi^3+i}-\frac{1}{\tau+\xi^3-(\tau^{1/3}+\xi)^3-i2^{-2k}\tau^{1/3}\xi}\bigg)\bigg|\lesssim \frac{\eta_{\leq 2k+5}(\mu)}{\mu^2},
\end{align*}
and when $|\tau^{1/3}+\xi|\sim 1$,
\begin{align*}
\Big|\big[\tau+\xi^3-(\tau^{1/3}+\xi)^3-i2^{-2k}\tau^{1/3}\xi\big]^{-1}\Big|=\Big|\big[-\tau^{1/3}\xi(3\tau^{1/3}+3\xi+i2^{-2k})\big]^{-1}\Big|\sim 2^{-2k}.
\end{align*}
Thus, we have
		$$|f(\xi,\tau)-f'(\xi,\tau)|\les 2^{k}|h(\tau)|\eta_{\leq 2k+5}(\mu)(\mu ^{-2}+2^{-2k}),$$
which implies from \eqref{Xk} that
		\begin{align}\label{Fff'}
			\|\FF(f-f')\|_Y&\les2^{sk}\|f-f'\|_{X_k}\les 2^{sk}\sum_{j\leq 2k+C }2^{j/2}2^{k}\|\eta_{j}(\mu )h(\tau)(\mu ^{-2}+2^{-2k})\|_{L_{\mu,\tau}^2}\nonumber\\
&\les 2^{sk}\sum_{j\leq 2k+C}2^{j/2}2^{k}2^{(j-2k)/2}(2^{-2j}+2^{-2k})\|h(\tau)\|_{L_{\tau}^2}\nonumber\\
&\les 2^{sk}\|h(\tau)\|_{L_\tau^2}.
		\end{align}
		Next, we consider $\|\FF(f')\|_Y$. Notice that 
\begin{align*}
\big\|\F^{-1}_{\xi}\big(\xi^{-1}\chi_{k}(\xi)\big)\big\|_{L^1_x}\sim 2^{-k},
\end{align*}
then by utilizing Young's inequality and performing the change of variables $\xi=\mu -\tau^{1/3}$, we obtain
		\begin{align}\label{Ff'}
		\|\FF f'\|_{Y}&\sim 2^k 2^{-k}\Big\|\int e^{it\tau-ix\tau^{1/3}}\chi_{[k-1,k+1]}(-\tau^{1/3})h(\tau)\tau^{-1/3}d\tau\cdot\int\frac{e^{ix\mu }\eta_0(\mu )}{3\mu +i2^{-2k}}d\mu \Big \|_{Y}\nonumber\\
&\les \Big\|\int e^{it\theta^3-ix\theta}\chi_{[k-1,k+1]}(-\theta)h(\theta^3)\theta d\theta\Big \|_{Y}\nonumber\\
&\les 2^{ks}\|\chi_{[k-1,k+1]}(-\theta)h(\theta^3)\theta \|_{L_\theta^2}\les 2^{ks}\|h\|_{L_{\tau}^2},
	\end{align}
where we used that the absolute value of the integral in $\mu$ is bounded. 	Therefore, combining \eqref{Fff'} and \eqref{Ff'}, we have obtained the desired conclusion \eqref{Ff}. 
\end{proof}
By Lemma \ref{estimate} and Lemma \ref{lemma2.4}, the proposition follows directly.
\begin{prop}\label{prop2.5}{\rm (Free solution estimates)} Let $k\in \Z_+$, $I\subset \R$ and $|I|\les 1$. For $f_k\in Z_k$, we have
\begin{align}
		&\|\FF(f_k)\|_{L_t^8L_x^8}\les\|f_k\|_{Z_k};\notag\\
&\|\FF(f_k)\|_{L_x^6L_t^6}\les 2^{-\frac{k}{6}}\|f_k\|_{Z_k}, \quad k\geq 1; \label{L6}\\
&\|\FF(f_k)\|_{L_x^{\infty}L_t^2}\les2^{-k}\|f_k\|_{Z_k},\quad k\geq 1;\label{Lt2}\\
&\|\FF(f_k)\|_{L_x^2L_{t\in I}^{\infty}}\les2^{\frac{3k}{4}}\|f_k\|_{Z_k};\label{Lx2}\\
&\|\FF(f_k)\|_{L_x^4L_t^{\infty}}\les2^{\frac{k}{4}}\|f_k\|_{Z_k}\label{Lx4}.
\end{align}
\end{prop}

\subsection{Localized trilinear estimates}
We now develop the $[4;Z]$ multiplier estimates adapted to the Airy resonance structure of the cubic nonlinearity, a fundamental step in establishing the subsequent nonlinear bounds. For $\xi_1,\xi_2,\xi_3\in \R$ and $\omega(\xi)=-\xi^3$, let
\begin{align*}
	\Omega(\xi_1,\xi_2,\xi_3)=\omega(\xi_1)+\omega(\xi_2)+\omega(\xi_3)-\omega(\xi_1+\xi_2+\xi_3),
\end{align*}
which is called the resonance function. For non-negative, compactly supported functions $f,g,h,u$, we define the multilinear functional
\begin{align*}
	J(f,g,h,u)=\,&\int_{\R^6} f(\xi_1,\tau_1) g(\xi_2,\tau_2) h(\xi_3,\tau_3)\\& u(\xi_1+\xi_2+\xi_3,\tau_1+\tau_2+\tau_3+\Omega(\xi_1,\xi_2,\xi_3)) d\xi_1d\xi_2d\xi_3d\tau_1d\tau_2d\tau_3.
\end{align*}
For convenience, given integers $k_1,k_2,k_3,k_4 \in \mathbb{Z}$, we denote by $k_{\text{max}}\geq k_{\text{sub}}\geq k_{\text{thd}}\geq k_{\text{min}}$ the maximum, second largest, third largest, and minimum values among $k_1,k_2,k_3,k_4$, respectively. 
Similarly, for nonnegative integers $j_1,j_2,j_3,j_4 \in \mathbb{Z}_+$, the quantities $j_{\text{max}}\geq j_{\text{sub}}\geq j_{\text{thd}}\geq j_{\text{min}}$ stand for the maximum, second largest, third largest, and minimum of $j_1,j_2,j_3,j_4$, respectively. 
$\mathbf{1}_A$ denotes the indicator function of set $A$.

\begin{lem}\label{4Zmultiplier}{\rm ($[4;Z]$ multiplier estimates)} Let $k_1,k_2,k_3,k_4 \in \mathbb{Z}$, $j_1,j_2,j_3,j_4 \in \mathbb{Z}_+$, and let $f_{k_i,j_i} \in L^2(\mathbb{R}^2)$ ($i=1,2,3$) be nonnegative functions supported in $D_{k_i,j_i}$. Define
$$
E = \left\| \mathbf{1}_{D_{k_4,j_4}}(\xi, \tau) \cdot (f_{k_1,j_1} \ast f_{k_2,j_2} \ast f_{k_3,j_3}) \right\|_{L^2},
$$
where $\ast$ denotes convolution on $\mathbb{R}^2$. Then the following estimates hold:

\item[\rm (a)] For any $k_1,k_2,k_3,k_4 \in \mathbb{Z}$ and $j_1,j_2,j_3,j_4 \in \mathbb{Z}_+$,
\begin{align}\label{a1}
E \lesssim 2^{\frac{k_{\rm{min}} + k_{\rm{thd}}}{2}} 2^{\frac{j_{\rm{min}} + j_{\rm{thd}}}{2}} \prod_{i=1}^3 \| f_{k_i,j_i} \|_{L^2}.
\end{align}

\item[\rm (b)] For any $k_1,k_2,k_3,k_4 \in \mathbb{Z}$, $j_1,j_2,j_3,j_4 \in \mathbb{Z}_+$, and $i \in \{1,2,3,4\}$, then
\begin{align}\label{c1}
E \lesssim 2^{\frac{j_1+j_2+j_3+j_4}{2}} 2^{-\frac{k_1+k_2+k_3+k_4}{6}} 2^{-\frac{j_i}{2} + \frac{k_i}{6}} \prod_{i=1}^3 \| f_{k_i,j_i} \|_{L^2}.
\end{align}

\item[\rm (c)] $k_{\rm{thd}} \leq k_{\rm{max}} - 10$. If there exists $i \in \{1,2,3,4\}$ such that $(k_i,j_i) = (k_{\rm{thd}}, j_{\rm{max}})$, then
  \begin{align}\label{b1}
  E \lesssim 2^{\frac{j_1+j_2+j_3+j_4}{2}} 2^{-\frac{j_{\rm{max}}}{2}} 2^{-k_{\rm{max}}} 2^{\frac{k_{\rm{thd}}}{2}} \prod_{i=1}^3 \| f_{k_i,j_i} \|_{L^2}.
  \end{align}
  Otherwise,
  \begin{align}\label{b2}
  E \lesssim 2^{\frac{j_1+j_2+j_3+j_4}{2}} 2^{-\frac{j_{\rm{max}}}{2}} 2^{-k_{\rm{max}}} 2^{\frac{k_{\rm{min}}}{2}} \prod_{i=1}^3 \| f_{k_i,j_i} \|_{L^2}.
  \end{align}

\item[\rm (d)] $k_{\rm{min}} \leq k_{\rm{max}} - 15$, then 
  \begin{align*}
  E \lesssim 2^{\frac{j_1+j_2+j_3+j_4}{2}} 2^{-\frac{11k_{\rm{max}}}{6}} 2^{-\frac{k_{\rm{min}}}{6}} \prod_{i=1}^3 \| f_{k_i,j_i} \|_{L^2}.
  \end{align*}
\end{lem}	
\begin{proof}
By duality, we have
\begin{align*}
E = \sup_{\| f \|_{L^2} = 1} \left| \int_{D_{k_4,j_4}} f \cdot (f_{k_1,j_1} \ast f_{k_2,j_2} \ast f_{k_3,j_3}) \, d\xi d\tau \right|.
\end{align*}
Let $f_{k_4,j_4} = \mathbf{1}_{D_{k_4,j_4}} \cdot f$, so $\| f_{k_4,j_4} \|_{L^2} \leq 1$. Define $f_{k_i,j_i}^{\#}(\xi, \tau) = f_{k_i,j_i}(\xi, \tau+\omega(\xi))$ for $i=1,2,3,4$, then these functions are supported in $I_{k_i} \times \bigcup_{|m| \leq 3} \tilde{I}_{j_i+m}$, and $\| f_{k_i,j_i}^{\#} \|_{L^2} = \| f_{k_i,j_i} \|_{L^2}$. A change of variables gives
\begin{align*}
\int_{D_{k_4,j_4}} f \cdot (f_{k_1,j_1} \ast f_{k_2,j_2} \ast f_{k_3,j_3}) \, d\xi d\tau = J(f_{k_1,j_1}^{\#}, f_{k_2,j_2}^{\#}, f_{k_3,j_3}^{\#}, f_{k_4,j_4}^{\#}).
\end{align*}
For simplicity, we omit the $\#$-superscript and only estimate $J = |J(f_{k_1,j_1}, f_{k_2,j_2}, f_{k_3,j_3}, f_{k_4,j_4})|$, with $f_{k_i,j_i}$ now supported in $I_{k_i} \times \tilde{I}_{j_i}$. The support of $f_{k_i,j_i}$ implies $J \equiv 0$ unless
\begin{align}\label{k}
|k_{\text{max}} - k_{\text{sub}}| \leq 5.
\end{align}
By the odd symmetry of $\omega(\xi) = -\xi^3$, we have the symmetry
\begin{align}\label{fghu}
|J(f,g,h,u)| = |J(g,f,h,u)| = |J(f,h,g,u)| = |J(\tilde{f},\tilde{g},u,h)|,
\end{align}
where $\tilde{f}(\xi,\tau) = f(-\xi,-\tau)$. Thus, for the sake of convenience, we may without loss of generality assume $k_1 \leq k_2 \leq k_3 \leq k_4$.

{\rm(a)} By the Cauchy-Schwarz inequality and support properties of $f_{k_i,j_i}$, we obtain
\begin{align*}
		J&\les 2^{\frac{j_{\rm{min}}+j_{\rm{thd}}}{2}}\int_{\R^3} \prod_{i=1}^{3}\| f_{k_i,j_i}(\xi_i,\tau_i) \|_{L_{\tau_i}^2}\cdot\|f_{k_4,j_4}(\xi_1+\xi_2+\xi_3,\tau) \|_{L_\tau^2}d\xi_1d\xi_2d\xi_3 \\
&\les 2^{\frac{j_{\rm{min}}+j_{\rm{thd}}}{2}}2^{\frac{k_{\rm{min}}+k_{\rm{thd}}}{2}}\prod_{i=1}^{3}\| f_{k_i,j_i} \|_{L^2},
	\end{align*}
where we use $\| f_{k_4,j_4} \|_{L^2} \leq 1$ and the fact that integrals over frequency intervals $I_{k_i}$ contribute factors of $2^{k_i/2}$.

{\rm (b)} Let $	f_{k_i,j_i}^{*}(\xi, \tau )=f_{k_i,j_i}(\xi, \tau -\omega(\xi))$, $i=1,2,3,4$. Plancherel's theorem  and the H\"older inequality imply that
	\begin{align*}
		J
		\les\| f^{*}_{k_1,j_1} \ast f^{*}_{k_2,j_2} \ast f^{*}_{k_3,j_3}\|_{L_{\xi,\tau}^2}\| f^{*}_{k_4,j_4} \|_{L_{\xi,\tau}^2}
		\les \prod_{i=1}^{3}\| \FF(f^{*}_{k_i,j_i})\|_{L_{x,t}^6}\|  f_{k_4,j_4} \|_{L_{\xi,\tau}^2}.
	\end{align*}
	By Minkowski's inequality and \eqref{L6}, we know that
	\begin{align*}
		\| \FF(f^{*}_{k_i,j_i})\|_{L_{x,t}^6}\les \int_{\R}\Big\|\int_{\R} f_{k_i,j_i}(\xi,\tau)\cdot e^{ix\xi} e^{-it\xi^3}d\xi\Big\|_{L_{x,t}^6}d\tau\les 2^{\frac{j_i}{2}}2^{-\frac{k_i}{6}} \| f_{k_i,j_i} \|_{L^2}.
	\end{align*}
Combining with \eqref{fghu}, we have the desired result \eqref{c1}.

{\rm (c)}   We assume $k_1 \leq k_2 \leq k_3 \leq k_4$, and analyze two cases based on whether $j_2 = j_{\rm{max}}$.

\textit{Case 1: } $j_2 \neq j_{\text{max}}$.
	If $j_4 = j_{\text{max}}$, according to \eqref{fghu} and Cauchy-Schwarz inequality, it suffices to prove that if $g_{i}:\R\to \R_{+} $ are $L^2$ functions supported in $I_{k_i}$ for $i=1,2,3$, and $g:\R^2\to \R_{+} $ is an $L^2$ function supported in $I_{k_4}\times\tilde{I}_{j_4}$, then
	\begin{align}\label{bb2}
		&\int_{\R^3}g_1(\xi_1)g_2(\xi_2)g_3(\xi_3)g(\xi_1+\xi_2+\xi_3,\Omega(\xi_1,\xi_2,\xi_3))d\xi_1d\xi_2d\xi_3\no \\\les &\ 2^{-k_{\text{max}}}2^{\frac{k_{\text{min}}}{2}}\prod_{i=1,2,3}\| g_i \|_{L^2(\R)}\| g \|_{L^2(\R^2)}.
	\end{align}
Indeed, by $k_2 \leq k_4 - 10$ and \eqref{k}, in the region $\{ |\xi_1| \sim 2^{k_1}, |\xi_2| \sim 2^{k_2},|\xi_3| \sim 2^{k_3}\}$, we have $|\xi_2-\xi_3|\sim |\xi_3|$ and $|\xi_2+\xi_3|\sim |\xi_3|$. Thus, we first apply the Cauchy-Schwarz inequality to the integrals in $\xi_2$ and $\xi_3$, and finally to the integral in $\xi_1$, which yields
	\begin{align}\label{case1a}
&\text{LHS}\eqref{bb2}\notag\\
\leq &\int_{|\xi_1|\sim 2^{k_1}}g_1(\xi_1)\|  g_2(\xi_2) \|_{L_{\xi_2}^2}\|\|  g_3(\xi_3) \|_{L_{\xi_3}^2}\| g(\xi_1+\xi_2+\xi_3,\Omega(\xi_1,\xi_2,\xi_3)) \|_{L_{\xi_2,\xi_3}^2}d\xi_1\no\\
\les &2^{\frac{k_{1}}{2}}\prod_{i=1,2,3}\| g_i \|_{L^2(\R)}\| g(\xi_1+\xi_2+\xi_3,\Omega(\xi_1,\xi_2,\xi_3)) \|_{L_{\xi_2,\xi_3}^2}\les \text{RHS}\eqref{bb2},
	\end{align}
where we use the change of variables $\xi=\xi_1+\xi_2+\xi_3$, $\mu=\Omega(\xi_1,\xi_2,\xi_3)$, and the Jacobian determinant
	\begin{align*}
		\Big |\frac{\partial(\xi, \mu)}{\partial(\xi_2,\xi_3)}\Big |=3|\xi_2-\xi_3||\xi_2+\xi_3|\sim 2^{2k_3}.
	\end{align*}
		
If $j_3=j_{\text{max}}$, then in view of $|k_3-k_4|\leq 5$ and the symmetry \eqref{fghu}, this case is identical to the one with $j_4 = j_{\text{max}}$.

If  $j_1=j_{\text{max}}$, similarly, it suffices to prove that if $g_{i}:\R\to \R_{+} $ are $L^2$ functions supported in $I_{k_i}$ for $i=2,3,4$, and $g:\R^2\to \R_{+} $ is an $L^2$ function supported in $I_{k_1}\times\tilde{I}_{j_1}$, then 
		\begin{align}\label{bb3}
			&\int_{\R^3}g_2(\xi_2)g_3(\xi_3)g_4(\xi_4)g(\xi_2+\xi_3+\xi_4,\Omega(\xi_2,\xi_3,\xi_4))d\xi_2d\xi_3d\xi_4\no \\\les &\ 2^{-k_{\text{max}}}2^{\frac{k_{\text{min}}}{2}}\prod_{i=2,3,4}\| g_i \|_{L^2(\R)}\| g \|_{L^2(\R^2)}.
		\end{align}
To be precise, let $\xi'_4=\xi_2+\xi_3+\xi_4$. Since $k_2\leq k_4-10$, in the integration region $\{|\xi'_4|\sim 2^{k_1},|\xi_2|\sim 2^{k_2},|\xi_3|\sim 2^{k_3}\}$, we then have
		\begin{align*}
			\Big|\frac{\partial \Omega(\xi_2,\xi_3,\xi'_4-\xi_2-\xi_3)}{\partial \xi_2}\Big|			=3|\xi_2-\xi_4||\xi_2+\xi_4|\sim2^{2k_4}.
		\end{align*}
		By using the Cauchy-Schwarz inequality, the left-hand side of \eqref{bb3} is bounded by
	\begin{align}\label{case1b}
&\int_{\R^2}g_3(\xi_3)\| g_2(\xi_2)g_4(\xi'_4-\xi_2-\xi_3)\|_{L_{\xi_2}^2}\|g(\xi'_4,\Omega(\xi_2,\xi_3,\xi'_4-\xi_2-\xi_3)) \|_{L_{\xi_2}^2}d\xi_3 d\xi'_4 \no\\
\les &\ 2^{-k_{\text{max}}}\int_{\R^2}g_3(\xi_3)\| g_2(\xi_2)g_4(\xi'_4-\xi_2-\xi_3)\|_{L_{\xi_2}^2}\|g(\xi'_4,\cdot) \|_{L_{(\cdot)}^2}d\xi_3 d\xi'_4 \no\\
\les &\ 2^{-k_{\text{max}}}\int_{\R}\| g_3(\xi_3) \|_{L_{\xi_3}^2}\| g_2(\xi_2)g_4(\xi'_4-\xi_2-\xi_3)\|_{L_{\xi_2,\xi_3}^2}\|g(\xi'_4,\cdot) \|_{L_{(\cdot)}^2} d\xi'_4 \no\\
\les&\ 2^{-k_{\text{max}}}2^{\frac{k_{\text{min}}}{2}}\| g_2 \|_{L^2(\R)}\| g_3\|_{L^2(\R)}\| g_4 \|_{L^2(\R)}\| g \|_{L^2(\R^2)}.
	\end{align}

\textit{Case 2: } $j_2 = j_{\text{max}}$. By \eqref{fghu} and Cauchy-Schwarz inequality, it suffices to show that if $g_{i}:\R\to \R_{+} $ are $L^2$ functions supported in $I_{k_i}$ for $i=1,3,4$, and $g:\R^2\to \R_{+} $ is an $L^2$ function supported in $I_{k_2}\times\tilde{I}_{j_2}$, then
	\begin{align}\label{bb1}
		&\int_{\R^3}g_1(\xi_1)g_3(\xi_3)g_4(\xi_4)g(\xi_1+\xi_3+\xi_4,\Omega(\xi_1,\xi_3,\xi_4))d\xi_1d\xi_3d\xi_4\no \\\les &2^{-k_{\text{max}}}2^{\frac{k_{\text{thd}}}{2}}\prod_{i=1,3,4}\| g_i \|_{L^2(\R)}\| g \|_{L^2(\R^2)}.
	\end{align}
If we adopt a method similar to that in \eqref{case1a}, we find that the Jacobian determinant 
\begin{align*}
		\Big |\frac{\partial(\xi_1+\xi_3+\xi_4, \Omega(\xi_1,\xi_3,\xi_4)
		}{\partial(\xi_3,\xi_4)}\Big |=3|\xi_3-\xi_4||\xi_3+\xi_4|
	\end{align*}
might be close to zero, making the left-hand side of \eqref{bb1} uncontrollable. Thus, we instead use a method analogous to that in \eqref{case1b}, where replacing $k_{\text{min}}$ with $k_{\text{thd}}$ is sufficient to obtain \eqref{bb1}.

\item[\rm (d)] Firstly, we consider the case $k_2\geq k_4-10$. It means $k_2\approx k_3\approx k_4$ from \eqref{k}. In this case we have  $|\xi_1+\xi_2|\sim2^{k_4}$, $|\xi_1+\xi_3|\sim2^{k_4}$ and  $|\xi_2+\xi_3|=|\xi-\xi_1|\sim2^{k_4}$. Then $|\Omega(\xi_1,\xi_2,\xi_3)|=3|\xi_1+\xi_2||\xi_1+\xi_3||\xi_2+\xi_3|\sim 2^{3k_4}$, which yields $j_{\text{max}}\geq 3k_4-5$.
For $j_1=j_{\text{max}}$, by \eqref{c1}, we have
\begin{align*}
	J\les2^{\frac{j_1+j_2+j_3+j_4}{2}}2^{-\frac{k_1+k_2+k_3+k_4}{6}}2^{-\frac{3k_4}{2}+\frac{k_1}{6}}\prod_{i=1}^{3}\| f_{k_i,j_i} \|_{L^2}\les2^{\frac{j_1+j_2+j_3+j_4}{2}} 2^{-2k_{\text{max}}}\prod_{i=1}^{3}\| f_{k_i,j_i} \|_{L^2}.
\end{align*}
For $j_1\neq j_{\text{max}}$, by \eqref{c1} and \eqref{k}, for some $i\in \{2,3,4\}$ such that $j_i=j_{\text{max}}$, we have
\begin{align*}
	J\les&2^{\frac{j_1+j_2+j_3+j_4}{2}} 2^{-\frac{k_1+k_2+k_3+k_4}{6}}2^{-\frac{3k_4}{2}+\frac{k_i}{6}}\prod_{i=1}^{3}\| f_{k_i,j_i} \|_{L^2}\\\les& 2^{\frac{j_1+j_2+j_3+j_4}{2}} 2^{-\frac{11k_{\text{max}}}{6}}2^{-\frac{k_1}{6}}\prod_{i=1}^{3}\| f_{k_i,j_i} \|_{L^2}.
\end{align*}
Secondly, we consider the case $k_2\leq k_4-10$ and $\xi_1\cdot\xi_2>0$. In this case we have $|\xi_1+\xi_2|\sim2^{k_2}$, $|\xi_1+\xi_3|\sim2^{k_3}$ and  $|\xi_2+\xi_3|\sim2^{k_3}$. Then  $|\Omega(\xi_1,\xi_2,\xi_3)|=3|\xi_1+\xi_2||\xi_1+\xi_3||\xi_2+\xi_3|\sim 2^{2k_3+k_2}$, which implies $j_{\text{max}}\geq 2k_3+k_2-5$. Combining with \eqref{b1} and \eqref{b2}, we have
	\begin{align*}
		J\les2^{\frac{j_1+j_2+j_3+j_4}{2}}2^{-\frac{2k_3+k_2}{2}}2^{\frac{k_2}{2}} 2^{-k_3}\prod_{i=1}^{3}\| f_{k_i,j_i} \|_{L^2}= 2^{\frac{j_1+j_2+j_3+j_4}{2}} 2^{-2k_{\text{max}}}\prod_{i=1}^{3}\| f_{k_i,j_i} \|_{L^2}.
	\end{align*}
Finally, we consider the case $k_2\leq k_4-10$ and $\xi_1\cdot\xi_2<0$. In this case, $|\xi_1+\xi_2|$ may fall near 0, and we need to perform a homogeneous dyadic decomposition on it.
If $j_4=j_{\text{max}}$, it suffices to prove that if $g_{i}:\R\to \R_{+} $ are $L^2$ functions supported in $I_{k_i}$ for $i=1,2,3$, and $g:\R^2\to \R_{+} $ is $L^2$ function supported in $I_{k_4}\times\tilde{I}_{j_4}$, then we have
	\begin{align}\label{bb4}
		&\int_{\R^3}g_1(\xi_1)g_2(\xi_2)g_3(\xi_3)g(\xi_1+\xi_2+\xi_3,\Omega(\xi_1,\xi_2,\xi_3))d\xi_1d\xi_2d\xi_3\no \\\les &2^{\frac{j_4}{2}}2^{-2k_{\text{max}}}\| g_1 \|_{L^2}\| g_2\|_{L^2}\| g_3 \|_{L^2}\| g \|_{L^2(\R^2)}.
	\end{align}     
	By localizing $|\xi_1+\xi_2|\sim2^l$ for $l\in \Z $,  and using $k_1\leq k_4-15$ along with \eqref{k}, we obtain $|\xi_1+\xi_3|\sim 2^{k_3}$ and $|\xi_2+\xi_3|\sim 2^{k_3}$. Consequently,  $|\Omega(\xi_1,\xi_2,\xi_3)|=3|\xi_1+\xi_2||\xi_1+\xi_3||\xi_2+\xi_3|\sim 2^{l+2k_3}$, which implies
	$l\leq j_4-2k_3+20$. Let $\xi'_1=\xi_1+\xi_2$, then in the integration region $\{|\xi'_1|\sim2^l,|\xi_2|\sim2^{k_2},|\xi_3|\sim2^{k_3}\}$, we have
\begin{align*}
	\Big|\frac{\partial(\xi_1'+\xi_3,\Omega(\xi'_1-\xi_2,\xi_2,\xi_3)) }{\partial (\xi_1',\xi_3)}\Big|=3|\xi_1+\xi_3||\xi_1-\xi_3|\sim 2^{2k_3}.
\end{align*}
 Using Cauchy-Schwarz inequality, the left-hand side of \eqref{bb4} is bounded by
	\begin{align*}			
    &\sum_{l\leq j_4-2k_3+20}\int_{\R^3}\chi_l(\xi_1+\xi_2)g_1(\xi_1)g_2(\xi_2)g_3(\xi_3)g(\xi_1+\xi_2+\xi_3,\Omega(\xi_1,\xi_2,\xi_3))d\xi_1d\xi_2d\xi_3\\
    \les &\sum_{l\leq j_4-2k_3+20}\int_{\R^3}\chi_l(\xi'_1) g_1(\xi'_1-\xi_2) g_2(\xi_2)g_3(\xi_3) g(\xi_1'+\xi_3,\Omega(\xi'_1-\xi_2,\xi_2,\xi_3))d\xi'_1d\xi_2d\xi_3\\\les &\sum_{l\leq j_4-2k_3+20}\int_{|\xi'_1|\sim2^l}\chi_l(\xi'_1)\cdot\|g_1(\xi_1'-\xi_2)\|_{L_{\xi_2}^2}\cdot \|g_3(\xi_3)\|_{L_{\xi_3}^2}\\&\qquad\qquad\quad\quad\cdot\|g_2(\xi_2)g(\xi_1'+\xi_3,\Omega(\xi'_1-\xi_2,\xi_2,\xi_3))\|_{L_{\xi_2,\xi_3}^2}d\xi'_1\\\les&\sum_{l\leq j_4-2k_3+20}2^{\frac{l}{2}}\|g_1\|_{L^2} \|g_3\|_{L^2}\|g_2\|_{L^2}\| g(\xi_1'+\xi_3,\Omega(\xi'_1-\xi_2,\xi_2,\xi_3))\|_{L_{\xi_3,\xi_1'}^2}
		\\\les&2^{\frac{j_4}{2}} 2^{-2k_{\text{max}}}\|  g_1 \|_{L^2} \| g_3\|_{L^2}\|g_2\|_{L^2} \|g \|_{L^2(\R^2)}.
	\end{align*}
	Since \eqref{k} and \eqref{fghu} hold,  the case $j_3=j_{\text{max}}$ is identical to the case $j_4=j_{\text{max}}$,  and the case $j_1=j_{\text{max}}$ is identical to the case $j_2=j_{\text{max}}$. Therefore, in the following, we only need to consider the case  $j_2=j_{\text{max}}$. In this case, it suffices to prove that if $g_{i}:\R\to \R_{+} $ are $L^2$ functions supported in $I_{k_i}$ for $i=1,3,4$, and $g:\R^2\to \R_{+} $ is $L^2$ function supported in $I_{k_2}\times\tilde{I}_{j_2}$, then we have
	\begin{align}\label{bb5}		&\int_{\R^3}g_1(\xi_1)g_3(\xi_3)g_4(\xi_4)g(\xi_1+\xi_3+\xi_4,\Omega(\xi_1,\xi_3,\xi_4))d\xi_1d\xi_3d\xi_4\no \\\les &2^{\frac{j_2}{2}}2^{-2k_{\text{max}}}\| g_1 \|_{L^2}\| g_3\|_{L^2}\| g_4 \|_{L^2}\| g \|_{L^2(\R^2)},
	\end{align}
    where we used the symmetry \eqref{fghu}. Similar to the previous case, where we needed to localize $|\xi_1+\xi_2| = |\xi_4-\xi_3|$, here we need to localize $|\xi_4+\xi_3|\sim2^l$, $l\in \Z$. Now we have $|\Omega(\xi_1,\xi_3,\xi_4)|=3|\xi_1+\xi_3||\xi_1+\xi_4||\xi_3+\xi_4|\sim 2^{l+2k_3}$, thus $l\leq j_2-2k_3+20$.
	Let $\xi'_3=\xi_3+\xi_4$, in the  integration region $\{|\xi_1|\sim2^{k_1},|\xi'_3|\sim2^{l},|\xi_4|\sim2^{k_4}\}$, we have
	\begin{align*}
		\Big|\frac{\partial(\xi_1+\xi_3',\Omega(\xi_1,\xi'_3-\xi_4,\xi_4)) }{\partial (\xi'_3,\xi_1)}\Big|=3|\xi_1-\xi_3||\xi_1+\xi_3|\sim 2^{2k_3}.
	\end{align*}
	 Applying the Cauchy-Schwarz inequality, the left-hand side of \eqref{bb5} is bounded by
	\begin{align*}		
		&\sum_{l\leq j_2-2k_3+20}\int_{\R^3}\chi_l(\xi'_3)g_1(\xi_1)g_3(\xi'_3-\xi_4)g_4(\xi_4)g(\xi_1+\xi_3',\Omega(\xi_1,\xi'_3-\xi_4,\xi_4))d\xi_1d\xi'_3d\xi_4\\\les&\sum_{l\leq j_2-2k_3+20}2^{\frac{l}{2}}\|g_1(\xi_1)\|_{L_{\xi_1}^2}\|g_3(\xi_3'-\xi_4)\|_{L_{\xi_4}^2}\|g_4(\xi_4)\|_{L_{\xi_4}^2} \|g(\xi_1+\xi_3',\Omega(\xi_1,\xi'_3-\xi_4,\xi_4))\|_{L_{\xi'_3,\xi_1}^2}
		\\\les&2^{\frac{j_2}{2}} 2^{-2k_{\text{max}}}\|  g_1 \|_{L^2} \| g_3\|_{L^2}\|g_4\|_{L^2} \|g \|_{L^2(\R^2)}.
	\end{align*}
	 Thus, we have completed the proof.
\end{proof}

\section{Local well-posedness}

In order to establish the local well‑posedness of the Cauchy problem \eqref{eq0}, we shall apply a contraction mapping argument to the truncated integral equation given below
\begin{align*}
	u(x,t)=\psi(t)W(t)u_0+\psi(t)\int_{0}^{t}W(t-t')(F(u,\bar{u}))(t')dt',
\end{align*}
where $W(t) := \mathcal{F}^{-1} e^{it \omega(\xi)}\mathcal{F}$, $\omega(\xi)=-\xi^3$, $\psi\in C_0^{\infty}(\R)$, $\operatorname{supp}\psi\subset[-2,2]$, $\psi\equiv 1$ on $[-1,1]$, and $$F(u,\bar{u})=\mathcal N_3(u,\bar u) +\mathcal N_5(u,\bar u).$$
In fact, by adopting the method developed in \cite{IonescuKenig07}, we readily derive the following linear estimates:
\begin{align}
	\|\psi(t)W(t)u_0\|_{W^s}&\les\|u_0\|_{H^s}, \quad s\geq 0;\label{linear1}\\
	\Big\|\psi(t)\int_{0}^{t}W(t-t')F(t')dt'\Big\|_{W^s}&\les \|F\|_{V^s}, \quad s\geq 0.\label{linear2}
\end{align}
Therefore, we need to obtain the trilinear and quintilinear estimates as follows
\begin{align*}
	 \|F(u,\bar{u})\|_{V^s}\lesssim \|u\|_{W^s}^K, \quad K=3\  \text{or}\ 5.
\end{align*}
To illustrate the main difficulty in the trilinear estimates, we write some representative cubic terms as 
\begin{equation}\label{cubic-expansion}
\mathcal N_3(u,\bar u) =c_1|u|^2u_{xx}+c_2u^2\bar{u}_{xx}+c_3\bar{u}\,u_x^2+c_4u\,u_x\bar{u}_x.
\end{equation}
The first two summands in \eqref{cubic-expansion} place both derivatives
on a single factor, whereas the remaining summands distribute the two
derivatives between two different factors. The former case, in which
both derivatives are concentrated on a single factor, is the most
delicate one. Nevertheless, since the resonance relation $\tau=-\xi^3$ is invariant under the transformation $(\tau,\xi)\mapsto(-\tau,-\xi)$, complex conjugations can be omitted
throughout the proofs of the multilinear estimates. Consequently, for the trilinear estimates, it suffices to consider the representative term $u^2u_{xx}$, while for the quintilinear estimates,
it suffices to consider $u^4u_x$.

\subsection{Trilinear estimates}

We first establish the dyadic trilinear estimates for each possible frequency configuration. 

\begin{prop}{\rm (low$*$low$*$low$\to$low)}\label{prop1}
	For $0 \leqslant k_1,k_2,k_3,k_4\leqslant99$, and $f_{k_i}\in Z_{k_i}$, $i=1,2,3$, we have
	\begin{align}\label{llll}
		2^{2k_1}\| \eta_{k_4}(\xi)(\tau-\omega(\xi)+i)^{-1}f_{k_1}\ast f_{k_2}\ast f_{k_3}\|_{Z_{k_4}}\les \prod_{i=1}^{3}\|f_{k_i} \|_{Z_{k_i}}.
	\end{align}
\end{prop}

\begin{proof} Let $f_{k_i,j_i}=f_{k_i}(\xi,\tau)\eta_{j_i}(\tau-\omega(\xi))$, $j_i\geq0$, $i=1,2,3$. By the $X_k$ norm, H\"older's inequality, and \eqref{a1}, we obtain
	\begin{align*}
		\text{LHS}\eqref{llll}&\lesssim \sum_{j_1,j_2,j_3,j_4\geq 0} 2^{\frac{j_{\rm{min}} + j_{\rm{thd}}}{2}} 2^{-\frac{j_4}{2}}(1+2^{\frac{j_4-3k_4}{2}})\prod_{i=1}^3 \| f_{k_i,j_i} \|_{L^2}\les \prod_{i=1}^{3}\|f_{k_i} \|_{Z_{k_i}},
	\end{align*}
where we used $|j_{\text{max}} - j_{\text{sub}}| \leq 5$ (which follows from the support properties). The proposition is proved.
\end{proof}

By symmetry, in the following discussion we may assume $k_2\geq k_3$. Now we consider the high$*$low$*$low$\to$high case, and further divide it into two subcases according to whether the two low frequencies are well separated, which correspond to the following proposition.

\begin{prop}{\rm (high$*$low$*$low$\to$high)}\label{hllh}
	Let $ k_1\geq100$, $|k_1-k_4|\leq5$, and $0\leq k_2,k_3\leq k_1-15$. Assume $f_{k_i} \in Z_{k_i}$ and that $\FF(f_{k_i})$ is supported in $\mathbb{R} \times I$ with $|I| \lesssim 1$ for $i = 1, 2, 3$.

\item[\rm (a)] If $k_3\leq k_2\leq k_3+5$, then
	\begin{align}\label{a3}
		2^{2k_1}\| \eta_{k_4}(\xi)(\tau-\omega(\xi)+i)^{-1}f_{k_1}\ast f_{k_2}\ast f_{k_3}\|_{Z_{k_4}}\les  2^{\frac{3k_2+3k_3}{4}}\prod_{i=1}^{3}\|f_{k_i} \|_{Z_{k_i}}.
	\end{align}
\item[\rm (b)] If $k_3\leq k_2-6$, then we have the improved estimate
\begin{align}\label{a4}
		2^{2k_1}\| \eta_{k_4}(\xi)(\tau-\omega(\xi)+i)^{-1}f_{k_1}\ast f_{k_2}\ast f_{k_3}\|_{Z_{k_4}}\les  2^{\frac{k_2+k_3}{4}}\prod_{i=1}^{3}\|f_{k_i} \|_{Z_{k_i}}.
	\end{align}
\end{prop}

\begin{proof} (a) We split the left-hand side of \eqref{a3} into three parts
	\begin{align*}
		\text{LHS}\eqref{a3}&\leq2^{2k_1}\| \eta_{k_4}(\xi)\eta_{\leq 2k_4-1}(\tau-\omega(\xi))(\tau-\omega(\xi)+i)^{-1}f_{k_1}\ast f_{k_2}\ast f_{k_3}\|_{Z_{k_4}}\\&\quad+2^{2k_1}\| \eta_{k_4}(\xi)\eta_{[2k_4,3k_4]}(\tau-\omega(\xi))(\tau-\omega(\xi)+i)^{-1}f_{k_1}\ast f_{k_2}\ast f_{k_3}\|_{Z_{k_4}}\\&\quad+2^{2k_1}\| \eta_{k_4}(\xi)\eta_{\geq 3k_4+1}(\tau-\omega(\xi))(\tau-\omega(\xi)+i)^{-1}f_{k_1}\ast f_{k_2}\ast f_{k_3}\|_{Z_{k_4}}\\&:=I+II+III.
	\end{align*}
	By the $Y_k$ norm, the H\"older's inequality, Lemma \ref{aaa} (a),(c), \eqref{Lt2}, and \eqref{Lx2}, we obtain
	\begin{align*}
		I&\les2^{2k_1}2^{-k_4}\|\FF(f_{k_1}\ast f_{k_2}\ast f_{k_3})\|_{L_x^1L_t^2}\\&\les2^{k_1}\|\FF     	(f_{k_1}) \|_{L_x^{\infty}L_t^2}\|\FF		(f_{k_2}) \|_{L_x^2L_{t\in I}^{\infty}}\|\FF		(f_{k_3}) \|_{L_x^2L_{t\in I}^{\infty}}
		\\&\les 2^{k_1}2^{-k_1}2^{\frac{3k_2+3k_3}{4}}\prod_{i=1}^{3}\|f_{k_i} \|_{Z_{k_i}}=2^{\frac{3k_2+3k_3}{4}} \prod_{i=1}^{3}\|f_{k_i} \|_{Z_{k_i}}.
	\end{align*}
For $II$, we use the	$X_k$ norm. For $j_4$ in the range $[2k_4,3k_4]$ the extra factor $(1+2^{(j_4-3k_4)/2})$ satisfies $1\le 1+2^{(j_4-3k_4)/2}\le 2$. By H\"older's inequality, and \eqref{Lt2}-\eqref{Lx4}, we get that
	\begin{align*}
		II&\les \sum_{2k_4\leq j_4\leq  3k_4}2^{2k_1}2^{-\frac{j_4}{2}}\|\FF(f_{k_1})\|_{L_x^{\infty}L_t^2}\|\FF(f_{k_2})\|_{L_x^{4}L_t^{\infty}}\|\FF(f_{k_3})\|_{L_x^{4}L_t^{\infty}}
		\\&\les 2^{2k_1}2^{-k_1}2^{-k_1}2^{\frac{k_2+k_3}{4}}\prod_{i=1}^{3}\|f_{k_i}\|_{Z_{k_i}}\les 2^{\frac{k_2+k_3}{4}}\prod_{i=1}^{3}\|f_{k_i}\|_{Z_{k_i}}.
	\end{align*}
	For the high-modulation case $III$, let $f_{k_i,j_i}=f_{k_i}(\xi,\tau)\eta_{j_i}(\tau-\omega(\xi)),$ $j_i\geq0,$ $i=1,2,3$. From the properties of the support, we have $ 1_{D_{k_4, j_4}}(\xi, \tau)(f_{k_1,j_1} \ast f_{k_2,j_2} \ast f_{k_3,j_3}) \equiv 0$ unless
\begin{align}\label{FCC}
2^{j_{\text{max}}}\sim {\max}\{2^{j_{\text{sub}}},|\Omega(\xi_1,\xi_2,\xi_3)|\}.
\end{align}
Since $|\Omega(\xi_1,\xi_2,\xi_3)|\ll 2^{3k_1}$ in the region $\{\xi_i\in \tilde{I}_{k_i}\}$,  we get $|j_{\text{max}}-j_{sub}|\leqslant 5$. Considering the worst case $(j_1,j_4)=(j_{\max},j_{\text{sub}})$ (the other cases are better) and applying Lemma \ref{4Zmultiplier} (a) together with Lemma \ref{aaa} (b), we have
	\begin{align}\label{worst0}
		III		&\les\sum_{\substack{j_1,j_4\geq 3k_4-5\\0\leq  j_2,j_3\leq j_1}}2^{2k_1}2^{-\frac{j_4}{2}}\bigl(1+2^{\frac{j_4-3k_4}{2}}\bigr)2^{\frac{j_2+j_3}{2}}2^{\frac{k_2+k_3}{2}}\prod_{i=1}^{3}\|f_{k_i,j_i} \|_{L^2}\no\\
&\les\sum_{j_1\geq 3k_4-5}2^{\frac {k_1}{2}}2^{\frac{k_2+k_3}{2}}2^{-\frac{j_1}{2}}2^{\frac{j_1}{2}}\|f_{k_1,j_1} \|_{L^2}\prod_{i=2}^{3}\|f_{k_i} \|_{Z_{k_i}}\no\\&\les2^{-k_1}2^{\frac{k_2+k_3}{2}}\prod_{i=1}^{3}\|f_{k_i} \|_{Z_{k_i}}\les \prod_{i=1}^{3}\|f_{k_i} \|_{Z_{k_i}}.
	\end{align}
		Thus, we complete the proof of (a).

(b)	Since the estimates for $II$ and $III$ in (a) already meet our requirements, we only need to re-estimate $I$. Furthermore, we decompose $f_{k_1}$ into high-modulation and low-modulation parts. For $I$, we write
\begin{align*}
		I&\leq2^{2k_1}\| \eta_{k_4}(\xi)\eta_{\leq 2k_4-1}(\tau-\omega(\xi))(\tau-\omega(\xi)+i)^{-1}f_{k_1}^h\ast f_{k_2}\ast f_{k_3}\|_{Z_{k_4}}\\&\quad +2^{2k_1}\| \eta_{k_4}(\xi)\eta_{\leq 2k_4-1}(\tau-\omega(\xi))(\tau-\omega(\xi)+i)^{-1}f_{k_1}^l\ast f_{k_2}\ast f_{k_3}\|_{Z_{k_4}}:=I_1+I_2,
	\end{align*} 
where $f_{k_1}^h=f_{k_1}(\xi,\tau)\eta_{\geq 2k_1+k_2-9}(\tau-\omega(\xi))$, $ f_{k_1}^l=f_{k_1}(\xi,\tau)\eta_{\leq 2k_1+k_2-10}(\tau-\omega(\xi))$.
Using the $Y_k$ norm, H\"older inequality, Lemma \ref{aaa} (a),(c), Plancherel identity, and \eqref{Lx4}, we can obtain
	\begin{align*}
		I_1
		&\les 2^{2k_1}2^{-k_4}\|\FF(f_{k_1}^h\ast f_{k_2}\ast f_{k_3})\|_{L_x^1L_t^2}\\&\les 2^{k_1}\|\FF (f_{k_1}^h)\|_{L_{x,t}^2}\|\FF (f_{k_2})\|_{L_x^4L_t^{\infty}}\|\FF (f_{k_3})\|_{L_x^4L_t^{\infty}}
		\\&\les 2^{k_1}2^{\frac{k_2+k_3}{4}}\|f_{k_1}^h\|_{L_{\xi,\tau}^2}\|f_{k_2}\|_{Z_{k_2}}\|f_{k_3}\|_{Z_{k_3}}.
	\end{align*}
	According to the definition of $X_k$ and $f_{k_1}^h$, we have
	\begin{align*}
		2^{k_1}\|f_{k_1}^h\|_{L_{\xi,\tau}^2}\les &\sum_{j\geq 2k_1+k_2-9} 2^{k_1}\|\eta_j(\tau-\omega(\xi))f_{k_1}(\xi,\tau)\|_{L_{\xi,\tau}^2}\\\les &\sum_{j\geq 2k_1+k_2-9}2^{\frac{j}{2}}\|\eta_j(\tau-\omega(\xi))f_{k_1}(\xi,\tau)\|_{L_{\xi,\tau}^2}\les \|f_{k_1}\|_{X_{k_1}}\les \|f_{k_1}\|_{Z_{k_1}}.
	\end{align*}
	Thus, we obtain $I_1\les \text{RHS}\eqref{a4}$.	For $I_2$, we set $f_{k_i,j_i}=f_{k_i}(\xi,\tau)\eta_{j_i}(\tau-\omega(\xi))$ ($j_i\geq0,$ $i=1,2,3$) and apply the $X_k$ norm
	\begin{align*}
		I_2\leq \sum_{\substack{j_4\leq 2k_4-1,\\j_1\leq 2k_1+k_2-10,\\j_2,j_3\geq 0}} 2^{2k_1}2^{-\frac{j_4}{2}}\| 1_{D_{k_4, j_4}}(\xi, \tau)(f_{k_1,j_1} \ast f_{k_2,j_2} \ast f_{k_3,j_3}) \|_{L_{\xi,\tau}^2}.
	\end{align*}
Because $|\xi_1+\xi_2|\sim2^{k_1}$, $|\xi_1+\xi_3|\sim2^{k_1}$, $|\xi_2+\xi_3|\sim2^{k_2}$, we have $|\Omega(\xi_1,\xi_2,\xi_3)|=3|\xi_1+\xi_2||\xi_1+\xi_3||\xi_2+\xi_3|\sim 2^{2k_1+k_2}$. From \eqref{FCC}, one of the following conditions holds
$$(1)\ j_2\approx j_3\geq 2k_1+k_2;\quad (2)\ j_2\approx 2k_1+k_2\geq j_3; \quad (3)\ j_3\approx 2k_1+k_2\geq j_2.$$	
	If condition (1) holds,  then by the definition of $X_k$, Lemma \ref{4Zmultiplier} (a) and Lemma  \ref{aaa}(b), we obtain 
	\begin{align*}
		I_2&\les \sum_{\substack{j_2\approx j_3\geq 2k_1+k_2,\\j_4\leq 2k_4-1,\\ j_1\leq 2k_1+k_2-10}}2^{2k_1}2^{-\frac{j_4}{2}}2^{\frac{j_1+j_4}{2}}2^{\frac{k_2+k_3}{2}}\prod_{i=1}^{3}\|f_{k_i,j_i}\|_{L^2}\\
&\les \sum_{j_2,j_3\geq 2k_1+k_2}k_12^{2k_1}2^{\frac{k_2+k_3}{2}}2^{\frac{3k_2}{2}-j_2}2^{\frac{3k_3}{2}-j_3}\prod_{i=1}^{3}\|f_{k_i} \|_{Z_{k_i}}\\
&\les k_12^{-2k_1}2^{2k_3}\prod_{i=1}^{3}\|f_{k_i}\|_{Z_{k_i}}\les 2^{\frac{k_2+k_3}{4}}\prod_{i=1}^{3}\|f_{k_i} \|_{Z_{k_i}},
	\end{align*}
where we used the fact that the factor $2^{\frac{j_i}{2}}(1+2^{\frac{j_i-3k_i}{2}})\sim 2^{j_i-\frac{3k_i}{2}}$ for $j_i\geq 3k_i$ ($i=2,3$).
For conditions (2) and (3), we only consider (2), as (3) is analogous. Using Lemma \ref{4Zmultiplier} (c) and Lemma \ref{aaa} (b), we can deduce that
	\begin{align*}
		I_2&\les \sum_{\substack{j_2\approx 2k_1+k_2,\\0\leq j_1,j_3\leq 2k_1+k_2,\\0\leq j_4\leq 2k_4-1}}2^{2k_1}2^{-\frac{j_4}{2}}2^{\frac{j_1+j_2+j_3+j_4}{2}}2^{-\frac{j_2}{2}}2^{-k_1}2^{\frac{k_2}{2}}\prod_{i=1}^{3}\|f_{k_i,j_i}\|_{L^2}\\
&\les \sum_{j_2\approx 2k_1+k_2}k_12^{k_1}2^{\frac{k_2}{2}}2^{\frac{3k_2}{2}-j_2}\prod_{i=1}^{3}\|f_{k_i} \|_{Z_{k_i}}\les 2^{\frac{k_2+k_3}{4}}\prod_{i=1}^{3}\|f_{k_i} \|_{Z_{k_i}}.
	\end{align*}
This completes the proof of \eqref{a4}.
\end{proof}

Now we consider the high$*$(sub-high)$*$low$\to$high case, which is easier to control than the high$*$low$*$low$\to$high case.

\begin{prop}{\rm ( high$*$(sub-high)$*$low$\to$high)}\label{hhlh}
	Let $ k_1\geqslant100$, $|k_1-k_4|\leqslant5$, $k_1-15\leqslant k_2\leqslant k_1$, and $0\leqslant k_3\leqslant k_1-15$. Assume  $f_{k_i}\in Z_{k_i}$ for $i=1,2,3$, we have
	\begin{align}\label{a6}
		2^{2k_1}\| \eta_{k_4}(\xi)(\tau-\omega(\xi)+i)^{-1}f_{k_1}\ast f_{k_2}\ast f_{k_3}\|_{Z_{k_4}}\lesssim k_1\cdot2^{\frac{k_1-k_3}{6}}\prod_{i=1}^{3}\|f_{k_i} \|_{Z_{k_i}}.
	\end{align}
\end{prop}

\begin{proof}  We split the left-hand side of \eqref{a6} into two parts
	\begin{align*}
		\text{LHS}\eqref{a6}&\leq2^{2k_1}\| \eta_{k_4}(\xi)\eta_{\leq 3k_1+20}(\tau-\omega(\xi))(\tau-\omega(\xi)+i)^{-1}f_{k_1}\ast f_{k_2}\ast f_{k_3}\|_{Z_{k_4}}\\&\quad+2^{2k_1}\| \eta_{k_4}(\xi)\eta_{\geq 3k_1+21}(\tau-\omega(\xi))(\tau-\omega(\xi)+i)^{-1}f_{k_1}\ast f_{k_2}\ast f_{k_3}\|_{Z_{k_4}}\\&:=I+II.
	\end{align*}
	Let $f_{k_i,j_i}=f_{k_i}(\xi,\tau)\eta_{j_i}(\tau-\omega(\xi)),$ $j_i\geq0,$ $i=1,2,3$. If $j_4\leq 3k_1+20$, we have $(1+2^{\frac{j_4-3k_4}{2}})\les 1$. By applying Lemma \ref{4Zmultiplier} (d), we obtain
	\begin{align}\label{k4l}
		I&\les\sum_{\substack{0\leq j_4\leq 3k_1+20,\\j_1,j_2,j_3\geq 0}}2^{2k_1}2^{-\frac{j_4}{2}}2^{\frac{j_1+j_2+j_3+j_4}{2}}2^{-\frac{11k_1}{6}}2^{-\frac{k_3}{6}}\prod_{i=1}^{3}\|f_{k_i,j_i}\|_{L^2}\no\\&\les k_1\cdot2^{\frac{k_1-k_3}{6}}\prod_{i=1}^{3}\|f_{k_i}\|_{Z_{k_i}}.
	\end{align}
If $j_4\geq 3k_1+21$, we have $2^{\frac{j_4}{2}}(1+2^{\frac{j_4-3k_4}{2}})\sim 2^{j_4-\frac{3k_4}{2}}$. From $|\Omega(\xi_1,\xi_2,\xi_3)|\les 2^{3k_1}$ and \eqref{FCC}, we know $|j_{\text{max}}-j_{sub}|\leqslant 5$. It suffices to consider the worst case $(j_1,j_4)=(j_{\max},j_{\text{sub}})$. Similarly to \eqref{worst}, applying Lemma \ref{4Zmultiplier} (a) yields $II\les \prod_{i=1}^{3}\|f_{k_i}\|_{Z_{k_i}}$. The proposition is now proved.
\end{proof}

Next we consider the case where the output frequency, namely $k_4$, is low.

\begin{prop}{\rm (high$*$high$*$low$\to$low)}\label{hhll}
	Let $ k_1\geq100$, $|k_1-k_2|\leq5$, and $0\leq k_3,k_4\leq k_1-15$. Assume $f_{k_i}\in Z_{k_i}$ and that $\FF(f_{k_i})$ is supported in $\R \times I$ with $|I|\les 1$ for $i=1,2,3$, we have
	\begin{align}\label{a5}
		&2^{2k_1}\| \eta_{k_4}(\xi)(\tau-\omega(\xi)+i)^{-1}f_{k_1}\ast f_{k_2}\ast f_{k_3}\|_{Z_{k_4}}\no\\
\les &\ k_1 \cdot 2^{\frac{7(k_1+k_2)}{12}+\frac{k_{\rm{thd}}}{2}-\frac{k_{\rm{min}}}{6}-\frac{3k_4}{2}}\prod_{i=1}^{3}\|f_{k_i} \|_{Z_{k_i}}.
	\end{align}
\end{prop}

\begin{proof} We split the left-hand side of \eqref{a5} into three parts
	\begin{align*}
		\text{LHS}\eqref{a5}&\leq2^{2k_1}\| \eta_{k_4}(\xi)\eta_{\leq 3k_4}(\tau-\omega(\xi))(\tau-\omega(\xi)+i)^{-1}f_{k_1}\ast f_{k_2}\ast f_{k_3}\|_{Z_{k_4}}\\&\quad+2^{2k_1}\| \eta_{k_4}(\xi)\eta_{[3k_4+1,2k_1+k_{\text{thd}}+20]}(\tau-\omega(\xi))(\tau-\omega(\xi)+i)^{-1}f_{k_1}\ast f_{k_2}\ast f_{k_3}\|_{Z_{k_4}}\\&\quad+2^{2k_1}\| \eta_{k_4}(\xi)\eta_{\geq 2k_1+k_{\text{thd}}+21}(\tau-\omega(\xi))(\tau-\omega(\xi)+i)^{-1}f_{k_1}\ast f_{k_2}\ast f_{k_3}\|_{Z_{k_4}}\\&:=I+II+III.
	\end{align*}
	Let $f_{k_i,j_i}=f_{k_i}(\xi,\tau)\eta_{j_i}(\tau-\omega(\xi)),$ $j_i\geqslant0,$ $i=1,2,3$. If $j_4\leq 3k_4$, we have $(1+2^{\frac{j_4-3k_4}{2}})\les 1$.  Similarly to \eqref{k4l}, applying Lemma \ref{4Zmultiplier} (d), we can obtain
$$
		I\les k_1\cdot2^{\frac{k_1-k_{\text{min}}}{6}}\prod_{i=1}^{3}\|f_{k_i}\|_{Z_{k_i}}.
$$
If $3k_4+1\leq j_4\leq 2k_1+k_{\text{thd}}+20$, we have $2^{\frac{j_4}{2}}(1+2^{\frac{j_4-3k_4}{2}})\sim 2^{j_4-\frac{3k_4}{2}}$. Then by applying Lemma \ref{4Zmultiplier} (d) and Lemma \ref{aaa} (b), we obtain
	\begin{align*}
		II&\les\sum_{\substack{j_4\leq 2k_1+k_{\text{thd}}+20,\\j_1,j_2,j_3\geq 0}}2^{2k_1}2^{-\frac{3k_4}{2}} 2^{\frac{j_1+j_2+j_3+j_4}{2}}2^{-\frac{11k	_1}{6}}2^{-\frac{k_{\text{min}}}{6}}\prod_{i=1}^{3}\|f_{k_i,j_i} \|_{L^2}\\
&\les 2^{\frac{7(k_1+k_2)}{12}}2^{\frac{k_{\text{thd}}}{2}-\frac{k_{\text{min}}}{6}}2^{-\frac{3k_4}{2}}\prod_{i=1}^{3}\|f_{k_i} \|_{Z_{k_i}}.
	\end{align*}
If $j_4\geq 2k_1+k_{\text{thd}}+21$, then $|\Omega(\xi_1,\xi_2,\xi_3)|\lesssim 2^{2k_1+k_{\text{thd}}}$ implies $|j_{\text{max}}-j_{\text{sub}}|\leqslant 5$. 
A direct calculation using Lemma \ref{4Zmultiplier} (c) yields
	\begin{align*}
		III&\les\sum_{\substack{j_{\max},j_{\text{sub}}\geq 2k_1+k_{\text{thd}}+15,\\j_{\text{thd}},j_{\min}\geq 0}}2^{2k_1}2^{-\frac{3k_4}{2}} 2^{\frac{j_1+j_2+j_3+j_4}{2}}2^{-\frac{j_{\max}}{2}}2^{-k_{\max}}2^{\frac{k_{\text{thd}}}{2}}\prod_{i=1}^{3}\|f_{k_i,j_i} \|_{L^2}\\
&\les\sum_{j_1,j_2,j_3\geq 0}2^{k_1+\frac{k_{\text{thd}}}{2}-\frac{3k_4}{2}} 2^{\frac{j_1+j_2+j_3}{2}}\prod_{i=1}^{3}\|f_{k_i,j_i} \|_{L^2}\\
&\les 2^{k_1+\frac{k_{\text{thd}}}{2}-\frac{3k_4}{2}}\prod_{i=1}^{3}\|f_{k_i} \|_{Z_{k_i}}.
	\end{align*}
Therefore, we complete the proof of the proposition.	
\end{proof}

\begin{prop}{\rm (high$*$high$*$(sub-high)$\to$low)}\label{hhhl}
	Let $ k_1\geqslant100$, $|k_1-k_2|\leqslant5$, $k_1-15\leqslant k_3\leqslant k_1$, and $0\leqslant k_4\leqslant k_1-15$. Assume  $f_{k_i}\in Z_{k_i}$ for $i=1,2,3$, we have
	\begin{align}\label{a9}
		2^{2k_1}\| \eta_{k_4}(\xi)(\tau-\omega(\xi)+i)^{-1}f_{k_1}\ast f_{k_2}\ast f_{k_3}\|_{Z_{k_4}}\lesssim  k_3\cdot 2^{\frac{3(k_1+k_2)}{4}}2^{-\frac{3k_4}{2}}\prod_{i=1}^{3}\|f_{k_i} \|_{Z_{k_i}}.
	\end{align}
\end{prop}
\begin{proof} Since $|\Omega(\xi_1,\xi_2,\xi_3)|=3|\xi_1+\xi_2||\xi_1+\xi_3||\xi_2+\xi_3|\sim 2^{3k_1}$ in this case, similar to Proposition \ref{hhll}, we split the left-hand side of \eqref{a9} into three parts
	\begin{align*}
		\text{LHS}\eqref{a9}&\leq2^{2k_1}\| \eta_{k_4}(\xi)\eta_{\leq 3k_4}(\tau-\omega(\xi))(\tau-\omega(\xi)+i)^{-1}f_{k_1}\ast f_{k_2}\ast f_{k_3}\|_{Z_{k_4}}\\&\quad+2^{2k_1}\| \eta_{k_4}(\xi)\eta_{[3k_4+1,3k_1+20]}(\tau-\omega(\xi))(\tau-\omega(\xi)+i)^{-1}f_{k_1}\ast f_{k_2}\ast f_{k_3}\|_{Z_{k_4}}\\&\quad+2^{2k_1}\| \eta_{k_4}(\xi)\eta_{\geq 3k_1+21}(\tau-\omega(\xi))(\tau-\omega(\xi)+i)^{-1}f_{k_1}\ast f_{k_2}\ast f_{k_3}\|_{Z_{k_4}}\\&:=I+II+III.
	\end{align*}
Let $f_{k_i,j_i}=f_{k_i}(\xi,\tau)\eta_{j_i}(\tau-\omega(\xi)),$ $j_i\geqslant0,$ $i=1,2,3$. Similarly to \eqref{k4l}, applying Lemma \ref{4Zmultiplier} (d), we can obtain
$$
		I\les k_3\cdot2^{\frac{k_1-k_{4}}{6}}\prod_{i=1}^{3}\|f_{k_i}\|_{Z_{k_i}}.
$$
For $II$, applying Lemma \ref{4Zmultiplier} (b), we get 
\begin{align*}
		II&\les\sum_{\substack{3k_4+1\leq j_4\leq 3k_1+20,\\j_1,j_2,j_3\geq 0}}2^{2k_1}2^{-\frac{3k_4}{2}} 2^{\frac{j_1+j_2+j_3}{2}}2^{-\frac{k_1+k_2+k_3}{6}}\prod_{i=1}^{3}\|f_{k_i,j_i} \|_{L^2}\\
&\les k_3\cdot 2^{\frac{3(k_1+k_2)}{4}}2^{-\frac{3k_4}{2}}\prod_{i=1}^{3}\|f_{k_i} \|_{Z_{k_i}}.
	\end{align*}
For $III$, we know that $|j_{\text{max}}-j_{\text{sub}}|\leqslant 5$. Without loss of generality, we consider $(j_1,j_4) = (j_{\max}, j_{\text{sub}})$. Applying Lemma \ref{4Zmultiplier} (a), we have
	\begin{align*}
		III		&\les\sum_{\substack{j_1\approx j_4\geq 3k_1+21\\0\leq  j_2,j_3\leq j_1}}2^{2k_1}2^{\frac{-3k_4}{2}}2^{\frac{j_2+j_3}{2}}2^{\frac{k_3+k_4}{2}}\prod_{i=1}^{3}\|f_{k_i,j_i} \|_{L^2}\no\\
&\les\sum_{j_1\geq 3k_1+21}2^{\frac {5k_1}{2}}2^{-k_4}2^{\frac{3k_1}{2}-j_1}2^{j_1-\frac{3k_1}{2}}\|f_{k_1,j_1} \|_{L^2}\prod_{i=2}^{3}\|f_{k_i} \|_{Z_{k_i}}\no\\&\les 2^{k_1-k_4}\prod_{i=1}^{3}\|f_{k_i} \|_{Z_{k_i}}.
	\end{align*}
This concludes the proof of the proposition.
\end{proof}

By a proof similar to that of Proposition \ref{hhlh}, we obtain the following proposition.

\begin{prop}{\rm (high$*$high$*$low$\to$(sub-high))}
	Let $ k_1\geqslant100$, $|k_1-k_2|\leqslant5$, $k_1-15\leqslant k_4\leqslant k_1$, and $0\leqslant k_3\leqslant k_1-15$. Assume  $f_{k_i}\in Z_{k_i}$ for $i=1,2,3$, we have
	\begin{align*}
		2^{2k_1}\| \eta_{k_4}(\xi)(\tau-\omega(\xi)+i)^{-1}f_{k_1}\ast f_{k_2}\ast f_{k_3}\|_{Z_{k_4}}\lesssim k_1\cdot 2^{\frac{k_1-k_3}{6}}\prod_{i=1}^{3}\|f_{k_i} \|_{Z_{k_i}}.
	\end{align*}
\end{prop}

\begin{prop}{\rm (high$*$high$*$high$\to$high)}\label{I+II}
	Let $ k_1\geq100$, $|k_1-k_2|\leq 5$, and $k_1-15\leq k_3,k_4\leq k_1$. Assume $f_{k_i}\in Z_{k_i}$ for $i=1,2,3$, we have
	\begin{align}\label{a2}
		2^{2k_1}\| \eta_{k_4}(\xi)(\tau-\omega(\xi)+i)^{-1}f_{k_1}\ast f_{k_2}\ast f_{k_3}\|_{Z_{k_4}}\les 2^{\frac{3k_1+3k_2}{4}}\prod_{i=1}^{3}\|f_{k_i} \|_{Z_{k_i}}.
	\end{align}
\end{prop}

\begin{proof}  We split the left-hand side of \eqref{a2} into two parts
	\begin{align*}
		\text{LHS}\eqref{a2}&\leq2^{2k_1}\| \eta_{k_4}(\xi)\eta_{\leq 3k_4+20}(\tau-\omega(\xi))(\tau-\omega(\xi)+i)^{-1}f_{k_1}\ast f_{k_2}\ast f_{k_3}\|_{Z_{k_4}}\\&\quad+2^{2k_1}\| \eta_{k_4}(\xi)\eta_{\geq 3k_4+21}(\tau-\omega(\xi))(\tau-\omega(\xi)+i)^{-1}f_{k_1}\ast f_{k_2}\ast f_{k_3}\|_{Z_{k_4}}\\&:=I+II.
	\end{align*}
	By using the $X_k$ norm, H\"older's inequality and \eqref{L6}, we have
	\begin{align*}
		I&\les\sum_{0\leq j_4\leq 3k_4+20}2^{2k_1}2^{-\frac{j_4}{2}}\|1_{D_{k_4, j_4}}(\xi, \tau)(f_{k_1} \ast f_{k_2} \ast f_{k_3}) \|_{L_{\xi,\tau}^2}  \\&\les2^{2k_1}\|\FF(f_{k_1}) \|_{L_x^6L_t^6}\|\FF(f_{k_2}) \|_{L_x^6L_t^6}\|\FF(f_{k_3}) \|_{L_x^6L_t^6}\\&\les 2^{2k_1}2^{-\frac{k_1+k_2+k_3}{6}}\prod_{i=1}^{3}\|f_{k_i} \|_{Z_{k_i}}\les 2^{\frac{3k_1+3k_2}{4}}\prod_{i=1}^{3}\|f_{k_i} \|_{Z_{k_i}}.
	\end{align*}
For $II$, $j_4\geq 3k_4+21$, we have $2^{\frac{j_4}{2}}(1+2^{\frac{j_4-3k_4}{2}})\sim 2^{j_4-\frac{3k_4}{2}}$. From $|\Omega(\xi_1,\xi_2,\xi_3)|\les 2^{3k_4}$ and \eqref{FCC}, we know $|j_{\text{max}}-j_{\rm sub}|\leqslant 5$. By symmetry, we may assume $(j_1,j_4) = (j_{\max}, j_{\text{sub}})$. Then, similarly to \eqref{worst0}, applying Lemma \ref{4Zmultiplier} (a) yields
\begin{align*}
		II&\les\sum_{\substack{j_1\approx j_4\geq 3k_4+15\\0\leq  j_2,j_3\leq j_1}}2^{2k_1}2^{-\frac{j_4}{2}}\bigl(1+2^{\frac{j_4-3k_4}{2}}\bigr)2^{\frac{j_2+j_3}{2}}2^{\frac{k_2+k_3}{2}}\prod_{i=1}^{3}\|f_{k_i,j_i} \|_{L^2}\no\\
&\les\sum_{j_1\geq 3k_4+15}2^{\frac {k_1}{2}}2^{\frac{k_2+k_3}{2}}2^{\frac{3k_1}{2}-j_1}2^{j_1-\frac{3k_1}{2}}\|f_{k_1,j_1} \|_{L^2}\prod_{i=2}^{3}\|f_{k_i} \|_{Z_{k_i}}\les \prod_{i=1}^{3}\|f_{k_i} \|_{Z_{k_i}}.
	\end{align*}
The proposition is now proved.
\end{proof}

\begin{thm}\label{thm1}
For $s\geqslant 3/4$, we have
\begin{align}\label{1}
	\|\partial_{xx}(\psi(t)u_1)\cdot(\psi(t)u_2)(\psi(t)u_3)\|_{V^s}\les&\|u_1\|_{W^{\frac{3}{4}}}\|u_2\|_{W^{\frac{3}{4}}} \|u_3\|_{W^s}+\|u_1\|_{W^{\frac{3}{4}}}\|u_2\|_{W^s}\|u_3\|_{W^{\frac{3}{4}}}\no\\&+\|u_1\|_{W^s}\|u_2\|_{W^{\frac{3}{4}}}\|u_3\|_{W^{\frac{3}{4}}}.
\end{align}
\end{thm}

\begin{proof}
Set \(f_{k_i}(\xi,\tau)=\eta_{k_i}(\xi)\mathcal{F}(\psi(t)u_i)(\xi,\tau)\) for \(i=1,2,3\). By the definition of the \(V^s\) norm,
\begin{align*}
&\|\partial_{xx}(\psi(t)u_1)\cdot(\psi(t)u_2)(\psi(t)u_3)\|_{V^s}^2 \\
\lesssim&\sum_{k_4=0}^{\infty}2^{2sk_4}
\Bigl(\sum_{k_1,k_2,k_3\in\mathbb{Z}_+}2^{2k_1}
\bigl\|\eta_{k_4}(\xi)(\tau-\omega(\xi)+i)^{-1}f_{k_1}*f_{k_2}*f_{k_3}\bigr\|_{Z_{k_4}}\Bigr)^2.
\end{align*}
Since the derivative falls on $u_1$, we may assume without loss of generality that $k_1 \geq k_2 \geq k_3$. Then we only need to control
\begin{align}\label{eight}
\sum_{k_1\ge k_2\ge k_3\ge 0}
2^{2k_1}\bigl\|\eta_{k_4}(\xi)(\tau-\omega(\xi)+i)^{-1}f_{k_1}*f_{k_2}*f_{k_3}\bigr\|_{Z_{k_4}}.
\end{align}
Given that $|k_{\text{max}} - k_{\text{sub}}| \leq 5$, we split the above sum into eight regions
\begin{align*}
\eqref{eight}\le \sum_{j=1}^{8}\sum_{(k_1,k_2,k_3,k_4)\in R_j}
2^{2k_1}\bigl\|\eta_{k_4}(\xi)(\tau-\omega(\xi)+i)^{-1}f_{k_1}*f_{k_2}*f_{k_3}\bigr\|_{Z_{k_4}},
\end{align*}
where
\begin{align*}
R_1&=\{0 \le k_1,k_2,k_3,k_4\le 99\},\\
R_2&=\{k_1\ge 100,\ |k_1-k_4|\le 5,\ 0\le k_2,k_3\le k_1-15,\ k_3\le k_2\le k_3+5\},\\
R_3&=\{k_1\ge 100,\ |k_1-k_4|\le 5,\ 0\le k_2,k_3\le k_1-15,\ k_3\le k_2-6\},\\
R_4&=\{k_1\ge 100,\ |k_1-k_4|\le 5,\ k_1-15\le k_2\le k_1,\ 0\le k_3\le k_1-15\},\\
R_5&=\{k_1\ge 100,\ |k_1-k_4|\le 5,\ k_1-15\le k_2,k_3\le k_1\},\\
R_6&=\{k_1\ge 100,\ |k_1-k_2|\le 5,\ 0\le k_3,k_4\le k_1-15\},\\
R_7&=\{k_1\ge 100,\ |k_1-k_2|\le 5,\ k_1-15\le k_3\le k_1,\ 0\le k_4\le k_1-15\},\\
R_8&=\{k_1\ge 100,\ |k_1-k_2|\le 5,\ k_1-15\le k_4\le k_1,\ 0\le k_3\le k_1-15\}.
\end{align*}
We now estimate each region using the previously established trilinear estimates. As an example, consider \((k_1,k_2,k_3,k_4)\in R_3\). By Proposition \ref{hllh} (b) (the improved estimate for well‑separated low frequencies) and H\"older's inequality,
\begin{align*}
\text{LHS of }\eqref{1}
&\lesssim \sum_{k_4=0}^{\infty}2^{2sk_4}
\Bigl(\sum_{(k_1,k_2,k_3,k_4)\in R_3}
2^{\frac{k_2+k_3}{4}}\prod_{i=1}^{3}\|f_{k_i}\|_{Z_{k_i}}\Bigr)^2\\
&\lesssim \sum_{k_1=0}^{\infty}2^{2sk_1}\|f_{k_1}\|_{Z_{k_1}}^2
\Bigl(\sum_{k_2=0}^{\infty}2^{\frac{k_2}{4}}\|f_{k_2}\|_{Z_{k_2}}\Bigr)^2
\Bigl(\sum_{k_3=0}^{\infty}2^{\frac{k_3}{4}}\|f_{k_3}\|_{Z_{k_3}}\Bigr)^2\\
&\lesssim \|u_1\|_{W^s}^2\;\|u_2\|_{W^{3/4}}^2\;\|u_3\|_{W^{3/4}}^2,
\end{align*}
where we used that \(\sum_{k}2^{k/4}\|f_k\|_{Z_k}\lesssim \|u\|_{W^{3/4}}\) (by Cauchy-Schwarz with the convergent sum \(\sum 2^{-k/2}\)). Now we use Proposition \ref{hhhl} to handle the region $R_7$, from the fact that $|k_1-k_2|\le5$ (so $k_2\approx k_1$), we obtain from Cauchy-Schwarz inequality
\begin{align*}
\text{LHS of }\eqref{1}&\lesssim \sum_{k_4=0}^{\infty} 2^{2sk_4} \Bigl( \sum_{(k_1,k_2,k_3,k_4)\in R_7} k_3 \cdot 2^{\frac{3(k_1+k_2)}{4}} 2^{-\frac{3k_4}{2}} 
\prod_{i=1}^{3}\|f_{k_i}\|_{Z_{k_i}}\Bigr)^2\\
&\lesssim \Big(\sum_{k_1=0}^{\infty}2^{\frac{3k_1}{2}}\|f_{k_1}\|_{Z_{k_1}}^2\Big)\Big(\sum_{k_2=0}^{\infty}2^{\frac{3k_2}{2}}\|f_{k_2}\|_{Z_{k_2}}^2\Big)\\&\qquad\cdot\Big(\sum_{k_4=0}^{\infty} 2^{2(\epsilon-1)k_4}\Big) \Big(\sum_{k_3=0}^{\infty} 2^{(s-\frac{1}{2})k_3}\|f_{k_3}\|_{Z_{k_3}}\Big)^2\\
&\lesssim \|u_1\|_{W^{3/4}}^2\;\|u_2\|_{W^{3/4}}^2\;\|u_3\|_{W^{s}}^2,
\end{align*}
where we used that $k_3\les 2^{\epsilon k_3}$ and $0<\epsilon< 1$. The remaining regions are handled analogously, so the details are omitted. This completes the proof of the theorem.
\end{proof}

\subsection{Quintilinear estimate}
\begin{thm}\label{thm3}
For $s\geq 1/4$, we have
\begin{align*}
	\|\partial_x(\psi(t)u_1)\cdot(\psi(t)^4 u_2u_3u_4u_5)\|_{V^s}\lesssim \prod_{i=1}^{5}\|u_i\|_{W^s}.
\end{align*}
\end{thm}

\begin{proof} Let $f_{k_i}(\xi,\tau)=\eta_{k_i}(\xi)\F(\psi(t)u_i)(\xi,\tau)$. By symmetry, we may assume that $k_2\geq k_3\geq k_4\geq k_5$. Thus by the definition of the $V^s$ norm, it suffices to estimate
$$\sum_{k_2\geq k_3\geq k_4\geq k_5}\sum_{k_1\in \Z_+}2^{k_1}\|\eta_{k_6}(\xi)(\tau-\omega(\xi)+i)^{-1}f_{k_1}\ast f_{k_2}\ast f_{k_3}\ast f_{k_4}\ast f_{k_5}\|_{Z_{k_6}}.$$

From the support properties, we known that $|k_{\max}-k_{\rm sub}|\leq 5$, where
$k_{\max}$ and $k_{\rm sub}$ denote the two largest numbers among
$k_1,\ldots,k_6$. When $0\leq k_i\leq 99$, $i=1,2,3,4,5,6$, by $X_k$ norm, H\"older inequality and Proposition \ref{prop2.5}, it follows directly from the same 
$X_k$ argument below that
\begin{align*}
	&2^{k_1}\|\eta_{k_6}(\xi)(\tau-\omega(\xi)+i)^{-1}f_{k_1}\ast f_{k_2}\ast f_{k_3}\ast f_{k_4}\ast f_{k_5}\|_{Z_{k_6}}	\nonumber\\
    \les&\sum_{j_6\geq 0} 2^{-{\frac{j_6}{2}}}\bigl(1+2^{\frac{j_6-3k_6}{2}}\bigr)2^{\frac{k_6}{2}}\|\FF(f_{k_1})\|_{L_{x}^\infty L_{t}^2}\prod_{i=2}^{5}\|\FF (f_{k_i})\|_{L_{x}^4 L_{t}^\infty}\nonumber\\
    \les& \prod_{i=1}^{5}\|f_{k_i}\|_{Z_{k_i}}.
\end{align*}
We therefore consider the worst case
\begin{equation}\label{worst}
  (k_1,k_2)=(k_{\max},k_{\rm sub}),\qquad
k_1\geq100,\qquad |k_1-k_2|\leq5.  
\end{equation}
The other cases are easier. If $k_1$ is the largest frequency and $k_6$ is the second largest one, the factor obtained from the output modulation gives a better estimate.  If $k_1$ is not the largest input frequency, we put a largest input factor in $L_x^\infty L_t^2$, and the additional factor $2^{k_1-k_{\max}}\leq1$ gives the desired estimate in
the same way. We divide the worst case \eqref{worst} into three parts
\begin{align*}
	&2^{k_1}\|\eta_{k_6}(\xi)(\tau-\omega(\xi)+i)^{-1}f_{k_1}\ast f_{k_2}\ast f_{k_3}\ast f_{k_4}\ast f_{k_5}\|_{Z_{k_6}}\\\les &2^{k_1}\|\eta_{k_6}(\xi)\eta_{\leq 2k_6}(\tau-\omega(\xi))(\tau-\omega(\xi)+i)^{-1}f_{k_1}\ast f_{k_2}\ast f_{k_3}\ast f_{k_4}\ast f_{k_5}\|_{Z_{k_6}}	\\
    &+2^{k_1}\|\eta_{k_6}(\xi)\eta_{[2k_6+1,3k_1+20]}(\tau-\omega(\xi))(\tau-\omega(\xi)+i)^{-1}f_{k_1}\ast f_{k_2}\ast f_{k_3}\ast f_{k_4}\ast f_{k_5}\|_{Z_{k_6}}\\
    &+2^{k_1}\|\eta_{k_6}(\xi)\eta_{\geq 3k_1+21}(\tau-\omega(\xi))(\tau-\omega(\xi)+i)^{-1}f_{k_1}\ast f_{k_2}\ast f_{k_3}\ast f_{k_4}\ast f_{k_5}\|_{Z_{k_6}}\\
    :=&I+II+III
\end{align*}
By using the H\"older inequality, \eqref{Lt2} and \eqref{Lx4}, we have
\begin{align*}
	I
	\les&2^{k_1}2^{-k_6}\| \FF(f_{k_1}\ast f_{k_2}\ast f_{k_3}\ast f_{k_4}\ast f_{k_5}) \|_{L_x^1L_t^2} \nonumber\\\les&2^{k_1} 2^{-k_6}\|\FF(f_{k_1})\|_{L_x^{\infty} L_t^2}\prod_{i=2}^{5}\|\FF (f_{k_i})\|_{L_x^4L_t^{\infty}}
	\nonumber\\\les&2^{-k_6}2^{\frac{k_2+k_3+k_4+k_5}{4}} \prod_{i=1}^{5}\|f_{k_i}\|_{Z_{k_i}}.
\end{align*}  	
By $X_k$ norm, H\"older's inequality, \eqref{Lt2} and \eqref{Lx4}, we have
\begin{align*}
	II\les&\sum_{j_6=2k_6+1}^{3k_1+20}2^{k_1}2^{-\frac{j_6}{2}}\bigl(1+2^{\frac{j_6-3k_6}{2}}\bigr)2^{\frac{k_6}{2}}\|\FF(f_{k_1} \ast f_{k_2} \ast f_{k_3}\ast f_{k_4}\ast f_{k_5}) \|_{L_x^1 L_{t}^2}\nonumber\\
    \les&(1+k_1)2^{-\frac{k_6}{2}} 2^{k_1} \|\FF(f_{k_1})\|_{L_x^\infty L_t^2}
\prod_{i=2}^5\|\FF(f_{k_i})\|_{L_x^4L_t^\infty}\nonumber\\
\les&(1+k_1)2^{-\frac{k_6}{2}} 2^{\frac{k_2+k_3+k_4+k_5}{4}}\prod_{i=1}^{5}\|f_{k_i}\|_{Z_{k_i}}.
\end{align*}  	
For $III$, $j_6\geq3k_1+21$, let
$f_{k_i,j_i}=\eta_{j_i}(\tau-\omega(\xi))f_{k_i}$.  From the support of
the convolution, we have $|\Omega_5(\xi_1,\ldots,\xi_5)|\lesssim2^{3k_1}$.
It follows that $|j_{\max}-j_{\rm sub}|\leq5$, where
$j_{\max}$ and $j_{\rm sub}$ denote the two largest numbers among
$j_1,\ldots,j_6$.  For the high-modulation part of $f_{k_i}$, set
$$
b_{i,j_i}=2^{\frac{j_i}{2}}\bigl(1+2^{\frac{j_i-3k_i}{2}}\bigr)
\|f_{k_i,j_i}\|_{L^2}.
$$
For a fixed modulation block, the proofs of \eqref{Lt2} and
\eqref{Lx4} give the more precise estimates
\begin{align*}
\|\FF(f_{k_i,j_i})\|_{L_x^\infty L_t^2}
&\lesssim 2^{-k_i}2^{\frac{j_i}{2}}\|f_{k_i,j_i}\|_{L^2}
=\frac{2^{-k_i}b_{i,j_i}}
{1+2^{(j_i-3k_i)/2}},\\
\|\FF(f_{k_i,j_i})\|_{L_x^4L_t^\infty}
&\lesssim 2^{\frac{k_i}{4}}2^{\frac{j_i}{2}}\|f_{k_i,j_i}\|_{L^2}
=\frac{2^{k_i/4}b_{i,j_i}}
{1+2^{(j_i-3k_i)/2}}.
\end{align*}
We divide the high-modulation contribution into two cases.  First,
assume that $j_6$ is one of the two largest
modulations.  Then, for some $r\in\{1,\ldots,5\}$,
$|j_6-j_r|\leq5$.  Denote the contribution of this case by
${III}_1$.  If $r=1$, we use the first refined estimate above for
$f_{k_1,j_1}$ and the $L_x^4L_t^\infty$ estimate for the other four
factors.  If $r\ne1$, we keep $f_{k_1}$ in
$L_x^\infty L_t^2$ and use the second refined estimate for
$f_{k_r,j_r}$.  Therefore, by Bernstein's inequality and H\"older's
inequality, in both cases we obtain
\begin{align*}
{III}_1
&\lesssim 2^{\frac{k_2+k_3+k_4+k_5}{4}}
\sum_{r=1}^5
\sum_{\substack{j_6\geq3k_1+21,\\ j_r\geq3k_1+15,\\
|j_6-j_r|\leq5}}
2^{-\frac{j_6}{2}}\bigl(1+2^{\frac{j_6-3k_6}{2}}\bigr)2^{\frac{k_6}{2}}
\frac{b_{r,j_r}}{1+2^{(j_r-3k_r)/2}}
\prod_{i\ne r}\|f_{k_i}\|_{Z_{k_i}}\nonumber\\
&\lesssim 2^{-k_6}2^{\frac{k_2+k_3+k_4+k_5}{4}}
\prod_{i=1}^5\|f_{k_i}\|_{Z_{k_i}},
\end{align*}
where we used $k_r\leq k_1$, $j_r\geq3k_1+15$.

Second, assume that the two largest modulations are both input
modulations, say $|j_r-j_l|\leq5$ for some $1\leq r<l\leq5$.  Put
$J=\max\{j_r,j_l\}$.  In this case $3k_1+21\leq j_6\leq J+5$.
Denoting this contribution by ${III}_2$, we have
\begin{align*}
{III}_2
&\lesssim 2^{-k_6}2^{\frac{k_2+k_3+k_4+k_5}{4}}
\sum_{1\leq r<l\leq5}
\sum_{\substack{j_r,j_l\geq3k_1+15,\\|j_r-j_l|\leq5}}
\frac{1+J-3k_1}
{\bigl(1+2^{(j_r-3k_r)/2}\bigr)
 \bigl(1+2^{(j_l-3k_l)/2}\bigr)}\nonumber\\
&\qquad\cdot b_{r,j_r}b_{l,j_l}
\prod_{i\ne r,l}\|f_{k_i}\|_{Z_{k_i}}\nonumber\\
&\quad\lesssim 2^{-k_6}2^{\frac{k_2+k_3+k_4+k_5}{4}}
\prod_{i=1}^5\|f_{k_i}\|_{Z_{k_i}},
\end{align*}
where we used $k_r,k_l\leq k_1$ and $j_r,j_l\geq J-5$,
$$
\frac{1+J-3k_1}
{\bigl(1+2^{(j_r-3k_r)/2}\bigr)
 \bigl(1+2^{(j_l-3k_l)/2}\bigr)}
\lesssim (1+J-3k_1)2^{-(J-3k_1)}\lesssim1.
$$
Combining the above three parts, we obtain
\begin{align}\label{quint-low-output}
&2^{k_1}\big\|\eta_{k_6}(\xi)(\tau-\omega(\xi)+i)^{-1}
f_{k_1}*\cdots*f_{k_5}\big\|_{Z_{k_6}}\nonumber\\
&\quad\lesssim(1+k_1)2^{-\frac{k_6}{2}} 2^{\frac{k_2+k_3+k_4+k_5}{4}}
\prod_{i=1}^5\|f_{k_i}\|_{Z_{k_i}}.
\end{align}
It remains to sum the dyadic frequencies.  Set
$$
a_{i,k}=2^{sk}\|\eta_k(\xi)\F(\psi(t)u_i)\|_{Z_k}.
$$
Since $s\geq1/4$ and $k_6\leq k_1+C$, \eqref{quint-low-output} implies
\begin{align*}
&2^{sk_6+k_1}\big\|\eta_{k_6}(\xi)(\tau-\omega(\xi)+i)^{-1}
f_{k_1}*\cdots*f_{k_5}\big\|_{Z_{k_6}}\\
&\quad\lesssim (1+k_1)2^{-\frac{k_6}{2}}
2^{s(k_6-k_1)}\prod_{i=1}^5a_{i,k_i}\\
&\quad\lesssim (1+k_1)2^{-c_sk_1}\prod_{i=1}^5a_{i,k_i},
\end{align*}
where $c_s=\min\{s,1/2\}>0$. Hence all the
dyadic sums converge. Applying Cauchy--Schwarz inequality, we conclude that
\begin{align*}
\|\partial_x(\psi(t)u_1)\cdot(\psi(t)^4u_2u_3u_4u_5)\|_{V^s}
\lesssim\prod_{i=1}^5\|u_i\|_{W^s}.
\end{align*}
This completes the proof.
\end{proof}

\subsection{Local well-posedness}
\begin{proof}[Proof of Theorem \ref{thm-lwp}] For $s\geq 3/4$ and the initial data $u_0\in H^s(\R)$, we define the following operator and set:
\begin{align*}
	&\varPhi(u)=\psi(t)W(t)u_0+\psi(t)\int_{0}^{t}W(t-t')\big(\mathcal N_3(u,\bar u) +\mathcal N_5(u,\bar u)\big)(t')dt',\\
	&\mathcal{B}=\{u\in W^s:\|u\|_{W^s}\leq 2C\|u_0\|_{H^s}\}.
\end{align*}
From \eqref{linear1}-\eqref{linear2}, and Theorem \ref{thm1}-\ref{thm3}, we have
\begin{align*}
	\|\varPhi(u)\|_{W^s}&\leq C \|u_0\|_{H^s}+C\|\mathcal N_3(u,\bar u) +\mathcal N_5(u,\bar u)\|_{V^s}\\&\leq C \|u_0\|_{H^s}+C(\|u\|_{W^s}^3+\|u\|_{W^s}^5)\\&\leq C\|u_0\|_{H^s}+C(4C^2\|u_0\|_{H^s}^2+16C^4\|u_0\|_{H^s}^4)\|u\|_{W^s}.
\end{align*}
Take $\|u_0\|_{H^s}$ sufficiently small such that $4C^3\|u_0\|_{H^s}^2+16C^5\|u_0\|_{H^s}^4\leq 1/4$, we obtain
\begin{align*}
		\|\varPhi(u)\|_{W^s}\leq 2C\|u_0\|_{H^s},
\end{align*}
which means $\varPhi(u)\subset \mathcal{B}$. Let $u,v\in \mathcal{B}$ with the same initial data $u_0$, then
\begin{align*}
	\|\varPhi(u)-\varPhi(v)\|_{W^s}\leq C\|u-v\|_{W^s}(\|u\|_{W^s}^2+\|v\|_{W^s}^2+\|u\|_{W^s}^4+\|v\|_{W^s}^4)\leq \frac12\|u-v\|_{W^s}.
\end{align*}
Therefore, $\varPhi$ defines a contraction mapping on $\mathcal{B}$. By the contraction mapping principle, it follows that there exists a unique solution $u \in \mathcal{B}$ such that $\varPhi(u) = u$.  

We next explain why it suffices to impose a smallness assumption only on
the $L^2$ norm of the initial data. Equation \eqref{eq0} is invariant
under the scaling
$$
u(x,t)\mapsto u_\lambda(x,t)
:=\lambda^{1/2}u(\lambda x,\lambda^3t),
\qquad \lambda>0,
$$
under which the initial data transform as
$$
u_{0,\lambda}(x)=\lambda^{1/2}u_0(\lambda x).
$$
A direct computation gives
$$
\|u_{0,\lambda}\|_{\dot H^s}
=\lambda^s\|u_0\|_{\dot H^s},
\qquad
\|u_{0,\lambda}\|_{L^2}
=\|u_0\|_{L^2}.
$$
Consequently,
$$
\|u_{0,\lambda}\|_{H^s}^2
\lesssim
\|u_0\|_{L^2}^2
+\lambda^{2s}\|u_0\|_{\dot H^s}^2.
$$
Thus, provided that
$$
\|u_0\|_{L^2}\ll1,
$$
we may choose $\lambda>0$ sufficiently small (for $s>0$) so that the
rescaled initial data satisfy
\begin{equation}\label{333}
\|u_{0,\lambda}\|_{H^s}\leq \varepsilon\ll1.
\end{equation}
We may therefore carry out the multilinear contraction argument for the
rescaled equation under the small data condition \eqref{333}, and then
recover the solution corresponding to the original initial data by
scaling back. This explains why the smallness assumption in
Theorem~\ref{thm-lwp} is required only on the scaling invariant quantity
$\|u_0\|_{L^2}$. This completes the proof.
\end{proof}

\section{A priori estimates in Sobolev spaces}
In this section, we establish conservation laws and derive a priori
estimates in Sobolev spaces for Equation \eqref{maineq}. To justify the
convergence of the series defining $\alpha(\kappa;u(t))$ in \eqref{alpha},
as well as term-by-term differentiation, we first need suitable estimates
for its leading term and the remainder. For convenience, we define 
$$  \Lambda(\kappa;u)=(\kappa-\partial_x)^{-1/2}u(\kappa+\partial_x)^{-1/2},
    \qquad \Gamma(\kappa;u)=(\kappa+\partial_x)^{-1/2}\bar u(\kappa-\partial_x)^{-1/2}.
$$
From Lemma 4.1 and Lemma 4.2 in \cite{KillipVisanZhang2018}, and Proposition 3.1 in \cite{ShanChenLuWang2024}, we have
\begin{lem}\label{lem4-1}
	Let $\kappa>0$ and $u\in\S(\mathbb{R})$. Then
    \begin{itemize}
\item[(i)] the imaginary part of the leading term in $\alpha(\kappa;u(t))$ satisfies 
	\begin{align}\label{alpha_1}
		\alpha_1(\kappa;u):={\rm Im} \big(i\kappa\, {\rm tr}\{(\kappa-\partial_x)^{-1}u(\kappa+\partial_x)^{-1}\bar{u}\}\big)=\int \frac{2\kappa^2|\hat{u}(\xi)|^2}{\xi^2+4\kappa^2}d\xi.
	\end{align}
\item[(ii)] the Hilbert-Schmidt norm of the remainder is as follows
\begin{align}\label{remainder1}
\|\Lambda(\kappa;u)\|^2_{\mathfrak I_2(\mathbb{R})}&=\|\Gamma(\kappa;u)\|^2_{\mathfrak I_2(\mathbb{R})}\sim \int_{\mathbb{R}} \log\bigl(4 + \tfrac{\xi^2}{\kappa^2}\bigr)\frac{|\hat u(\xi)|^2\,d\xi}{\sqrt{4\kappa^2 + \xi^2}}.
\end{align}
\end{itemize}
\end{lem}
By Lemma~\ref{lem4-1}, for sufficiently large $\kappa$, the series defining
$\alpha(\kappa;u)$ converges and can be differentiated term by term. We now establish the conservation of the
perturbation determinant $\alpha(\kappa;u(t))$.
\begin{prop}\label{conser0}
	Let $u \in C^{\infty}(\mathbb{R}, \mathcal{S})$ be a Schwartz solution to \eqref{maineq}. Then, choosing $\kappa$ sufficiently large so that the right-hand side of
\eqref{remainder1} is less than $1/2$, we obtain
	\begin{align*}
		\frac{d}{dt}\alpha(\kappa;u(t))=0.
	\end{align*}
\end{prop}
\begin{proof} For convenience, we denote $R_+ := (\partial_x + \kappa)^{-1}$ and $R_- := (\partial_x - \kappa)^{-1}$. Since the series defining $\alpha(\kappa;u(t))$ converges and can be differentiated term by term, and using the equation \eqref{maineq}, we obtain
	\begin{align*}
		\frac{d}{dt}\alpha(\kappa;u)=\,&\sum_{m=1}^{\infty}(-i\kappa)^m{\rm tr}\{(R_-uR_+\bar{u} )^{m-1}(R_-u_tR_+\bar{u}+R_-uR_+\bar{u}_t)\}\\=\,&\sum_{m=1}^{\infty}(-i\kappa)^m{\rm tr}\biggl\{(R_-uR_+\bar{u} )^{m-1}\bigg[R_-\Big(u_{xxx}+3i(|u|^2u_x)_x-\frac{3}{2}(|u|^4u)_x\Big)R_+\bar{u}\\&+R_-uR_+\Big (\bar{u}_{xxx}-3i(|u|^2\bar{u}_x)_x-\frac{3}{2}(|u|^4\bar{u})_x\Big )\bigg]\biggr\}\\=:&A+B+C.
	\end{align*}
	where
	\begin{align*}
		A&=\sum_{m=1}^{\infty}(-i\kappa)^m{\rm tr}\{(R_-uR_+\bar{u} )^{m-1}(R_-u_{xxx}R_+\bar{u}+R_-uR_+\bar{u}_{xxx})\}=:\sum_{m=1}^{\infty}A_m,\\
		B&=\sum_{m=1}^{\infty}3i(-i\kappa)^m{\rm tr}\{(R_-uR_+\bar{u} )^{m-1}(R_-(|u|^2u_x)_xR_+\bar{u}-R_-uR_+(|u|^2\bar{u}_x)_x)\}=:\sum_{m=1}^{\infty}B_m,\\
		C&=\sum_{m=1}^{\infty}-\frac{3}{2}(-i\kappa)^m{\rm tr}\{(R_-uR_+\bar{u} )^{m-1}(R_-(|u|^4u)_xR_+\bar{u}+R_-uR_+(|u|^4\bar{u})_x)\}=:\sum_{m=1}^{\infty}C_m.
	\end{align*}
	It suffices to show that
	\begin{align*}
		A_1=0,\qquad A_2+B_1=0,
	\end{align*}
	and
	\begin{align}\label{cancel1}
		A_m+B_{m-1}+C_{m-2}=0, \qquad m\geq 3.
	\end{align}
	
First, we consider $A_1$. From the Fourier transform, we have
\begin{align*}
	{\rm tr}(R_+\bar{u}R_-u_{xxx})=-{\rm tr}(R_+\bar{u}_{xxx}R_-u).
\end{align*}
Combining this identity with the cyclic property of the trace, $\operatorname{tr}(AB)=\operatorname{tr}(BA)$, we can obtain that
\begin{align*}
	A_1=-i\kappa {\rm tr}\{R_+\bar{u}R_-u_{xxx}+R_+\bar{u}_{xxx}R_-u\}=0.
\end{align*}

Next, we show that for $m \geqslant 3$,  \eqref{cancel1} holds. Since $u_x=\partial_x u-u\partial_x$, we can write
	\begin{align*}
		u_{xxx}&=-u(\partial_x^3-3\kappa\partial_x^2+3\kappa^2\partial_x+3\kappa^3)+(\partial_x^3+3\kappa\partial_x^2+3\kappa^2\partial_x-3\kappa^3)u\\&\quad-(\partial_x-\kappa)(3u_x+6\kappa u)(\partial_x+\kappa),\\
		\bar{u}_{xxx}&=(\partial_x^3-3\kappa\partial_x^2+3\kappa^2\partial_x+3\kappa^3)\bar{u}-\bar{u}(\partial_x^3+3\kappa\partial_x^2+3\kappa^2\partial_x-3\kappa^3)\\&\quad-(\partial_x+\kappa)(3\bar{u}_x-6\kappa\bar{u})(\partial_x-\kappa).
	\end{align*}
The second identity follows from the first by complex conjugation, together with the replacement $\kappa\mapsto -\kappa$. Substituting these identities into $A_{m}$, the first two terms in each identity cancel pairwise. The remaining terms, together with the cyclic property of the trace, yield
	\begin{align}\label{AM}
		A_m=\,&-3(-i\kappa)^m{\rm tr}\{(R_-uR_+\bar{u} )^{m-2}(R_-uR_+\bar{u}^2(u_x+2\kappa u)+R_-u(\bar{u}_x-2\kappa \bar{u})uR_+\bar{u})\}\no\\
		=\,&-3(-i\kappa)^m{\rm tr}\{(R_-uR_+\bar{u} )^{m-2}(R_-uR_+\bar{u}^2u_x+R_-u^2\bar{u}_xR_+\bar{u})\}\no\\
		&-6\kappa(-i\kappa)^m{\rm tr}\{(R_-uR_+\bar{u} )^{m-2}(R_-uR_+|u|^2\bar{u}-R_-|u|^2uR_+\bar{u})\}.
	\end{align}
	
For $B_m$, we note that
	\begin{align*}
	&(|u|^2u_x)_x=(\partial_x-\kappa)(|u|^2u_x)-(|u|^2u_x)(\partial_x+\kappa)+2\kappa|u|^2u_x,\\
	&(|u|^2\bar{u}_x)_x=(\partial_x+\kappa)(|u|^2\bar{u}_x)-(|u|^2\bar{u}_x)(\partial_x-\kappa)-2\kappa|u|^2\bar{u}_x.
\end{align*}
Substituting the first term of $(|u|^2u_x)_x$ and the second term of $(|u|^2\bar{u}_x)_x$ into $B_m$, and cycling the trace, we obtain
\begin{align}\label{Bm1}
	&{\rm tr}\{(R_-uR_+\bar{u})^{m-1}(R_-(\partial_x-\kappa)|u|^2u_xR_+\bar{u}+R_-uR_+|u|^2\bar{u}_x(\partial_x-\kappa))\}\no\\=\,&\ {\rm tr}\{(R_-uR_+\bar{u})^{m-2}[R_-uR_+(|u|^2\bar{u}u_x+|u|^2\bar{u}_xu)R_+\bar{u}]\}\no\\
	=\,&\ \frac{1}{2}{\rm tr}\{(R_-uR_+\bar{u})^{m-2}[R_-uR_+(|u|^4)_xR_+\bar{u}]\}.
\end{align}
Similarly, substituting the second term of $(|u|^2 u_x)_x$ and the first term of $(|u|^2 \bar{u}_x)_x$ into $B_m$, we obtain
\begin{align}\label{Bm2}
	&{\rm tr}\{(R_-uR_+\bar{u})^{m-1}(R_-|u|^2u_x(\partial_x+\kappa)R_+\bar{u}+R_-uR_+(\partial_x+\kappa)|u|^2\bar{u}_x)\}\no\\=\,&{\rm tr}\{(R_-uR_+\bar{u})^{m-1}(R_-|u|^2u_x\bar{u}+R_-u|u|^2\bar{u}_x)\}\no\\=\,&\frac{1}{2}{\rm tr}\{(R_-uR_+\bar{u})^{m-1}(R_-(|u|^4)_x)\}.
\end{align}
Moreover, using
\begin{align*}
	(|u|^4)_x=(\partial_x+\kappa)|u|^4-|u|^4(\partial_x+\kappa),\qquad (|u|^4)_x=(\partial_x-\kappa)|u|^4-|u|^4(\partial_x-\kappa),
\end{align*}
and plugging these identities into \eqref{Bm1} and \eqref{Bm2}, together with the cyclic property of the trace, we obtain
\begin{align*}
	&\text{LHS}\eqref{Bm1}-\text{LHS}\eqref{Bm2}\\=\,&\frac{1}{2}{\rm tr}\{(R_-uR_+\bar{u})^{m-2}(R_-u|u|^4R_+\bar{u}-R_-uR_+|u|^4\bar{u}-R_-uR_+|u|^4\bar{u}+R_-|u|^4uR_+\bar{u})\}\\=\,&{\rm tr}\{(R_-uR_+\bar{u})^{m-2}(R_-u|u|^4R_+\bar{u}-R_-uR_+|u|^4\bar{u})\}.
\end{align*}
Therefore, 
\begin{align}\label{BM}
	B_{m}=\,&3i(-i\kappa)^{m}{\rm tr}\{(R_-uR_+\bar{u} )^{m-2}(R_-u|u|^4R_+\bar{u}-R_-uR_+|u|^4\bar{u})\}\no\\
	&-6(-i\kappa)^{m+1} {\rm tr}\{(R_-uR_+\bar{u})^{m-1}(R_-|u|^2u_xR_+\bar{u}+R_-uR_+|u|^2\bar{u}_x)\}
\end{align}

For $C_m$, we rewrite
	\begin{align*}
		&(|u|^4u)_x=(\partial_x-\kappa)(|u|^4u)-(|u|^4u)(\partial_x+\kappa)+2\kappa|u|^4u,\\
		&(|u|^4\bar{u})_x=(\partial_x+\kappa)(|u|^4\bar{u})-(|u|^4\bar{u})(\partial_x-\kappa)-2\kappa|u|^4\bar{u},
	\end{align*}
The first two terms in each identity cancel each other. Thus,
\begin{align}\label{CM}
	C_m=3\kappa (-i\kappa)^{m}{\rm tr}\{(R_-uR_+\bar{u})^{m-1}(-R_-u|u|^4R_+\bar{u}+R_-uR_+|u|^4\bar{u})\}.
\end{align}
We observe that the term $C_{m-2}$ cancels the first term in $B_{m-1}$. Therefore, we may rearrange the terms \eqref{AM} and \eqref{BM}-\eqref{CM} to obtain
	\begin{align*}
	&A_m+B_{m-1}+C_{m-2}\\
	=\,&-3(-i\kappa)^m{\rm tr}\{(R_-uR_+\bar{u} )^{m-2}(R_-uR_+(\bar{u}^2u_x+2|u|^2\bar{u}_x)+R_-(u^2\bar{u}_x+2|u|^2u_x)R_+\bar{u})\}\\
	&-6\kappa(-i\kappa)^m{\rm tr}\{(R_-uR_+\bar{u} )^{m-2}(R_-uR_+|u|^2\bar{u}-R_-|u|^2uR_+\bar{u})\}.
	\end{align*}
It is worth noting that
\begin{align}
	&\bar{u}^2u_x+2|u|^2\bar{u}_x=(u\bar{u}^2)_x=(\partial_x+\kappa)u\bar{u}^2-u\bar{u}^2(\partial_x-\kappa)-2\kappa u\bar{u}^2,\label{no1}\\
	&u^2\bar{u}_x+2|u|^2u_x=(u^2\bar{u})_x=(\partial_x-\kappa)u^2\bar{u}-u^2\bar{u}(\partial_x+\kappa)+2\kappa u^2\bar{u}.\label{no2}
\end{align}
The first two terms in each identity cancel once again, and hence
	\begin{align*}
	&A_m+B_{m-1}+C_{m-2}\\
		=\,&-3(-i\kappa)^m {\rm tr}\{(R_-uR_+\bar{u})^{m-2}(-2\kappa R_-uR_+u\bar{u}^2+2\kappa R_-u^2\bar{u}R_+\bar{u})\}\\
		&-6\kappa(-i\kappa)^m{\rm tr}\{(R_-uR_+\bar{u} )^{m-2}(R_-uR_+\bar{u}^2u-R_-u^2\bar{u}R_+\bar{u})\}=0.
\end{align*}

Finally, we consider $A_2 + B_1$. In $B_1$, the contributions from \eqref{Bm1} and \eqref{Bm2} vanish because ${\rm tr}(R_{\pm}[\partial_x,f])=0$. Then from \eqref{no1}-\eqref{no2}, we have that
	\begin{align*}
	A_2+B_1=\,&3\kappa^2{\rm tr}\{R_-uR_+\bar{u}^2u_x+R_-u^2\bar{u}_xR_+\bar{u}\}+6\kappa^3{\rm tr}\{R_-uR_+\bar{u}|u|^2-R_-|u|^2uR_+\bar{u}\}\\
	&+3\kappa^2{\rm tr}\{R_-(2|u|^2u_x)R_+\bar{u}+R_-uR_+(2|u|^2\bar{u}_x)\}=0.
\end{align*}
This completes the proof.
\end{proof}

For sufficiently large $\kappa>0$ and $L\in\mathbb N$, from \eqref{log-a}, we define 
\begin{equation*}
\varphi_L(\kappa;u):=
\operatorname{Im}\left[
-\alpha(\kappa;u)
-\sum_{j=0}^{2L+1}\frac{E_j(u)}{(i\kappa)^j}
\right].
\end{equation*}
Since the real part of $E_1(u)$ is zero, we have 
\begin{equation*}
\varphi_0(\kappa;u):=
\operatorname{Im}\left[
-\alpha(\kappa;u)
-E_0(u)
\right]=-\operatorname{Im}
\alpha(\kappa;u)+\frac{1}{2}\int_{\mathbb R}|u|^2\,dx.
\end{equation*}
From Proposition \ref{conser0} and the conservation of mass, we know that $\varphi_0$ is conserved:
\begin{align}\varphi_0(\kappa;u(t))=\varphi_0(\kappa;u_0), \quad \forall t\in \R.\label{varphicons}\end{align}
From \eqref{alpha} and \eqref{alpha_1}, we know that the quadratic part of $\varphi_0(\kappa;u)$ is 
\begin{equation*}
\varphi_{0,0}(\kappa;u)=\frac{1}{2}\int \frac{\xi^2|\hat{u}(\xi)|^2}{\xi^2+4\kappa^2}d\xi.
\end{equation*}

Before proving Theorem \ref{thm-apriori}, we present several important estimates that will be needed. 
From the definition of $\Lambda$ and Bernstein inequality, we find that 
\begin{align}\label{op1}
    \|\Lambda(\kappa;u_N)\|_{\operatorname{op}}\lesssim \frac{1}{\kappa}\|u_N\|_{L^\infty}\lesssim \frac{\sqrt{N}}{\kappa}\|u_N\|_{L^2}.
\end{align}
From \eqref{remainder1}, we know that 
\begin{align}\label{I2}
    \|\Lambda(\kappa;u_N)\|^2_{\mathfrak I_2(\mathbb{R})}\sim  \frac{\log(4+N^2/\kappa^2)}{\kappa+N}\|u_N\|_{L^2}^2
    \lesssim
    \begin{cases}
        \kappa^{-1}\|u_N\|_{L^2}^2, & N\le\kappa,\\[1mm]
        N^{-1}\log(4+N^2/\kappa^2), & N>\kappa.
    \end{cases}
\end{align}

The following lemma collects the estimates of the quadratic term.
\begin{lem} Let $u\in\mathcal S(\mathbb R)$ and $0<s<1$. Then, for every $K>0$,
\begin{align}
 \|u_{>K}\|_{ H^s(\mathbb R)}^2&\lesssim  \int_K^\infty \kappa^{2s-1}\varphi_{0,0}(\kappa;u)\,d\kappa\lesssim  \|u\|_{\dot H^s(\mathbb R)}^2, \label{quadratic1}\\
  \|u\|_{\dot H^s(\mathbb R)}^2&\lesssim_s\int_K^\infty
\kappa^{2s-1}\varphi_{0,0}(\kappa;u)\,d\kappa+K^{2s}\|u\|_{L^2(\mathbb R)}^2.\label{quadratic2}
\end{align}
\end{lem}
\begin{proof} A direct computation shows that
 $$\int_{K}^{\infty} \frac{\kappa^{2s-1}}{\xi^2 + 4\kappa^2} \, d\kappa=\xi^{2s-2}\int_{K/\xi}^{\infty} \frac{\eta^{2s-1}}{1 + 4\eta^2} \, d\eta\sim (\xi^2 + K^2)^{s-1}.$$
 Therefore, 
$$\int_K^\infty \kappa^{2s-1}\varphi_{0,0}(\kappa;u)\,d\kappa\sim \int \frac{\xi^2|\hat{u}(\xi)|^2}{(\xi^2 + K^2)^{1-s}}d\xi,$$
which immediately yields \eqref{quadratic1} and \eqref{quadratic2}. 
\end{proof}
The following lemma collects the nonlinear remainder estimates from \cite[Lemma 2.3]{BLP2023}. 
\begin{lem} Let $u\in\mathcal S(\mathbb R)$ and $0<s<1$. Then there exists $K_0=K_0(\|u\|_{H^{1/3}(\mathbb R)})\ge1$ such that for $ \beta=\max\left\{\frac{s+1}{4},\frac13\right\}$,
\begin{equation}\label{spectral-remainder}
|\varphi_0(\kappa;u)-\varphi_{0,0}(\kappa;u)|
\le
\frac{C(s,\|u\|_{H^\beta(\mathbb R)})}{\kappa^{s+1}}
\bigl(\|u\|_{H^s(\mathbb R)}+1\bigr),
\qquad \kappa\ge K_0.
\end{equation}
\end{lem}
When absorbing the remainder terms, we need to use the equicontinuity in $L^2$, and this result is established in Theorem 1.5 \cite{ HGKV2023} for the entire dNLS hierarchy.
\begin{lem}\label{equicontinuous}
Let $Q \subseteq \mathcal{S}(\mathbb{R})$ be $L^2$-bounded and equicontinuous. Then
\begin{align*}
Q_{*} = \left\{ u \in \mathcal{S}(\mathbb{R}) : a(\lambda; u) \equiv a(\lambda; \tilde{u}) \ \text{for some} \ \tilde{u} \in Q \right\} 
\end{align*}
is also $L^2$-bounded and equicontinuous.
\end{lem}

Fix $s>0$ and $M>0$, and set
$$   Q_{s,M}:=\bigl\{f\in\mathcal S(\mathbb R):\|f\|_{H^s}\le M\bigr\}.$$ 
Since
$$ \sup_{f\in Q_{s,M}}\|P_{>N}f\|_{L^2}\lesssim M N^{-s},$$ 
the set $Q_{s,M}$ is $L^2$-bounded and equicontinuous.
Lemma~4.5 therefore implies that its spectral level set $(Q_{s,M})_*$ is also $L^2$-bounded and equicontinuous. 

Let $u$ be a solution of \eqref{maineq} with
$u_0\in Q_{s,M}$, and let $I\ni0$ be an interval on which $u(t)\in\mathcal S(\mathbb R)$.
By Proposition~4.2, $ a(\lambda;u(t))=a(\lambda;u_0),\,t\in I$. Thus $u(t)\in(Q_{s,M})_*$ for every $t\in I$, and hence 
 \begin{align}\label{equicontinuous1}
    \lim_{N\to\infty}\sup_{t\in I} \|P_{>N}u(t)\|_{L^2}=0.
\end{align}
The frequency cutoff can be chosen uniformly over all such solutions with initial data in $Q_{s,M}$.


\begin{proof}[Proof of Theorem \ref{thm-apriori}] (1) {\bf Case} $0\leq s<1/2$. The case $s=0$ corresponds to mass conservation, so next we consider $0<s<1/2$. We can choose $K_0>1$ large enough such that for $\kappa>K\geq K_0$, $\alpha(\kappa;u)$ is well defined. First, we carry out a controlled estimate of the remainder term. From \eqref{alpha}, we know that
\begin{align*}
    &\int_K^\infty 
      \kappa^{2s-1} |\varphi_0(\kappa;u)-\varphi_{0,0}(\kappa;u)| d\kappa \lesssim \int_K^\infty\kappa^{2s+1} \|\Lambda(\kappa;u)\Gamma(\kappa;u)\|_{\mathfrak I_2}^2d\kappa.
 \end{align*}
 We perform the frequency decomposition of the integrand to obtain
\begin{align*}
    \|\Lambda(\kappa;u)\Gamma(\kappa;u)\|_{\mathfrak I_2}^2 \lesssim & \sum_{N_1\sim N_2\ge N_3\ge N_4} \|\Lambda(\kappa;u_{N_1})\|_{\mathfrak I_2}\|\Lambda(\kappa;u_{N_2})\|_{\mathfrak I_2}\|\Lambda(\kappa;u_{N_3})\|_{\operatorname{op}}\|\Lambda(\kappa;u_{N_4})\|_{\operatorname{op}}\notag\\
    \lesssim & \sum_{R_j} \|\Lambda(\kappa;u_{N_2})\|_{\mathfrak I_2}^2\|\Lambda(\kappa;u_{N_3})\|_{\operatorname{op}}\|\Lambda(\kappa;u_{N_4})\|_{\operatorname{op}},
\end{align*}
where we decompose the summation into the regions
\begin{align*}
    R_1&=\{N_2\le K\},\\
    R_2&=\{K<N_2\le\kappa,\ N_3\le\eta K\},\quad   R_3=\{K<N_2\le\kappa,\ N_3>\eta K\},\\
    R_4&=\{N_2>\kappa,\ N_3\le\eta K\},\quad     R_5=\{N_2>\kappa,\ N_3>\eta K\}.
\end{align*}
Let $I_j$ denote the contribution of $R_j$ after integration against $\kappa^{2s+1}d\kappa$. Notice that the cutoff $\eta K$ is independent of the
integration variable $\kappa$. 

We use  \eqref{op1} and \eqref{I2} to estimate $I_1$-$I_5$. On $R_1$, given $0<s<1/2$, we have
\begin{align*}
    I_1 &\lesssim \int_K^\infty\kappa^{2s-2}d\kappa\cdot K\|u\|_{L^2}^4 \lesssim_s K^{2s}\|u\|_{L^2}^4.
\end{align*}
On $R_2$, both lower frequencies are at most $\eta K$. Thus for $0<s<1/2$,
\begin{align*}
    I_2  &\lesssim \eta K \|u\|_{L^2}^2\int_K^\infty\kappa^{2s-2}d\kappa\cdot  \sum_{N_2>K}N_2^{-2s}N_2^{2s}\|u_{N_2}\|_{L^2}^2\lesssim_s
      \eta  \|u\|_{L^2}^2  \|u_{>K}\|_{H^s}^2.
\end{align*}
On $R_3$,  for $0<s<1/2$, we obtain
\begin{align*}
    I_3 &\lesssim  \|u\|_{L^2}\|u_{>\eta K}\|_{L^2} \sum_{N_2>K} N_2\|u_{N_2}\|_{L^2}^2\cdot\int_{N_2}^\infty\kappa^{2s-2}d\kappa
        \lesssim_s \|u\|_{L^2}\|u_{>\eta K}\|_{L^2}\|u_{>K}\|_{H^s}^2.
\end{align*}

For the remaining two regions, we use
\begin{align*}
    \int_K^{N_2}
    \kappa^{2s-1}\log(4+N_2^2/\kappa^2) \,d\kappa  &=  N_2^{2s}\int_{K/N_2}^1 r^{2s-1}\log(4+r^{-2})\,dr \lesssim_s N_2^{2s},
\end{align*}
which is valid because $s>0$. Therefore,
\begin{align*}
    I_4 &\lesssim \eta  \|u\|_{L^2}^2 \sum_{N_2>K} \frac{K}{N_2}\|u_{N_2}\|_{L^2}^2\cdot \int_K^{N_2}  \kappa^{2s-1}\log(4+N_2^2/\kappa^2) \,d\kappa \lesssim_s  \eta  \|u\|_{L^2}^2 \|u_{>K}\|_{H^s}^2.\\
    I_5  &\lesssim  \|u\|_{L^2}\|u_{>\eta K}\|_{L^2} \sum_{N_2>K} \|u_{N_2}\|_{L^2}^2\cdot \int_K^{N_2}  \kappa^{2s-1}\log(4+N_2^2/\kappa^2) \,d\kappa\\
       &\lesssim_s \|u\|_{L^2}\|u_{>\eta K}\|_{L^2}\|u_{>K}\|_{H^s}^2.
\end{align*}
Summing the five contributions and using \eqref{quadratic1}, we find that for $0<\eta<1$ and $\eta K\ge1$,
\begin{align}\label{small1}
    &\int_K^\infty 
      \kappa^{2s-1} |\varphi_0(\kappa;u)-\varphi_{0,0}(\kappa;u)| d\kappa \notag\\
     \lesssim_s  & K^{2s}\|u\|_{L^2}^4+\left(\eta\|u\|_{L^2}^2 +\|u\|_{L^2}\|u_{>\eta K}\|_{L^2}\right)\|u_{>K}\|_{H^s}^2\notag\\
    \lesssim_s  & K^{2s}\|u_0\|_{L^2}^4+\left(\eta\|u_0\|_{L^2}^2 +\|u_0\|_{L^2}\|u_{>\eta K}\|_{L^2}\right)\int_K^\infty 
      \kappa^{2s-1} \varphi_{0,0}(\kappa;u) d\kappa.
 \end{align}
We may choose $\eta\in(0,1)$ small so that 
$$
    C_s\eta \|u_0\|_{L^2}^2\le\frac14.
$$
By Lemma \ref{equicontinuous} and \eqref{equicontinuous1}, we may then choose $K$ large enough such that
$$
    C_s\|u_0\|_{L^2}\sup_{t\in I}\|u_{>\eta K}\|_{L^2}\le\frac14.
$$
Therefore, absorbing the last term in \eqref{small1} and using \eqref{varphicons} give that
\begin{align*}
\int_K^\infty \kappa^{2s-1} \varphi_{0,0}(\kappa;u(t)) d\kappa &\leq 2\int_K^\infty \kappa^{2s-1} \varphi_0(\kappa;u(t)) d\kappa + 2C_sK^{2s}\|u_0\|_{L^2}^4\notag\\
&= 2\int_K^\infty \kappa^{2s-1} \varphi_0(\kappa;u_0) d\kappa + 2C_sK^{2s}\|u_0\|_{L^2}^4\notag\\
&\leq 3 \int_K^\infty \kappa^{2s-1} \varphi_{0,0}(\kappa;u_0) d\kappa + 4C_sK^{2s}\|u_0\|_{L^2}^4.
 \end{align*}
It follows from \eqref{quadratic1} and \eqref{quadratic2} that 
\begin{align}\label{result1}
\sup_{t}\|u(t)\|_{H^s(\mathbb R)}^2
\le C\bigl(s,\|u_0\|_{H^s(\mathbb R)}\bigr), \quad  \forall s\in[0,\frac12).
\end{align}

(2) {\bf Case} $1/2\leq s<1$. With (1) established, we may choose $K_0=K_0(\|u_0\|_{H^{1/3}})$ so that the remainder estimate \eqref{spectral-remainder} applies.  From \eqref{quadratic2} and \eqref{spectral-remainder}, for every $t\in \R$ and $ \beta=\frac{s+1}{4}\in [\frac{3}{8},\frac12)$, we have
\begin{align*}
\|u(t)\|_{\dot H^s(\mathbb R)}^2
&\le  C(s) \int_{K_0}^{\infty} \kappa^{2s-1}|\varphi_{0,0}(\kappa;u(t))|\,d\kappa + C(s)K_0^{2s}\|u(t)\|_{L^2(\mathbb R)}^2 \notag\\
&\le C(s)\int_{K_0}^{\infty} \kappa^{2s-1}|\varphi_0(\kappa;u(t))|\,d\kappa + C\bigl(s,K_0,\|u(t)\|_{H^\beta(\mathbb R)}\bigr) \bigl(\|u(t)\|_{H^s(\mathbb R)}+1\bigr)\notag\\
&\le C(s)\int_{K_0}^{\infty}\kappa^{2s-1}|\varphi_0(\kappa;u_0)|\,d\kappa+C\bigl(s,\|u_0\|_{H^s(\mathbb R)}\bigr)\bigl(\|u(t)\|_{H^s(\mathbb R)}+1\bigr)
\end{align*}
where we used \eqref{varphicons} and the result of \eqref{result1}. Utilizing \eqref{spectral-remainder} and \eqref{quadratic1} again, we have 
\begin{align*}
\|u(t)\|_{\dot H^s(\mathbb R)}^2
&\le C(s)\int_{K_0}^{\infty}\kappa^{2s-1}|\varphi_{0,0}(\kappa;u_0)|\,d\kappa+C\bigl(s,\|u_0\|_{H^s(\mathbb R)}\bigr)\bigl(\|u(t)\|_{H^s(\mathbb R)}+1\bigr)\notag\\
&\le C\bigl(s,\|u_0\|_{H^s(\mathbb R)}\bigr)+C\bigl(s,\|u_0\|_{H^s(\mathbb R)}\bigr)\|u(t)\|_{H^s(\mathbb R)}
\end{align*}
Conservation of mass gives
\begin{align*}
\|u(t)\|_{H^s(\mathbb R)}^2 &\le \|u_0\|_{L^2(\mathbb R)}^2 +\|u(t)\|_{\dot H^s(\mathbb R)}^2\\
&\le C\bigl(s,\|u_0\|_{H^s(\mathbb R)}\bigr)+C\bigl(s,\|u_0\|_{H^s(\mathbb R)}\bigr)\|u(t)\|_{H^s(\mathbb R)}\\
&\le \frac12\|u(t)\|_{H^s(\mathbb R)}^2
+C\bigl(s,\|u_0\|_{H^s(\mathbb R)}\bigr).
\end{align*}
Absorbing the resulting term yields
$$
\sup_{t}\|u(t)\|_{H^s(\mathbb R)}^2
\le C\bigl(s,\|u_0\|_{H^s(\mathbb R)}\bigr), \quad  \forall s\in[\frac12,1).
$$
This completes the proof.
\end{proof}

\section{Global well-posedness}

\begin{proof}[Proof of Theorem \ref{globalwell}]  Recall that local well-posedness of \eqref{maineq2} in $H^s(\mathbb R)$ for $s\geq 3/4$, which is equivalent to that of \eqref{maineq} via the gauge transformation, was proved in \cite{ChenChenXiao2026}. Let $u_0\in H^s(\mathbb R)$, and let $  u\in C([0,T^*);H^s(\mathbb R))$ be the maximal lifespan solution established by the local well-posedness theory. We shall prove that $T^*=\infty$.

The a priori estimates in Theorem \ref{thm-apriori} are established for Schwartz solutions. We now extend them to solutions in $H^s(\mathbb R)$ by a standard density argument.  Let $u_0^{(n)}\in\mathcal S(\mathbb R)$ satisfy $ u_0^{(n)}\rightarrow u_0 $ in $H^s(\mathbb R)$. Fix a compact interval $J\subset I_{\max}=[0,T^*)$. By local
well-posedness, persistence of regularity, and continuous dependence, the corresponding smooth solutions $u^{(n)}$ are defined on $J$ for all sufficiently large $n$ and satisfy $u^{(n)}\rightarrow u$ in $C(J;H^s(\mathbb R))$. Applying the a priori estimate to $u^{(n)}$ and passing to the limit gives
\begin{align}\label{aprioriz}
\sup_{t\in J}\|u(t)\|_{H^s} \le C_s\bigl(\|u_0\|_{H^s}\bigr),  \qquad J\subset I_{\max}. 
\end{align}
Since $J\subset I_{\max}$ is arbitrary, the estimate \eqref{aprioriz} holds on $I_{\max}$.

Fix $3/4\le s<1$ and let $u_0\in H^s(\mathbb R)$. Suppose that $T^*<\infty$.  \eqref{aprioriz} implies
$$ \sup_{0\le t<T^*}\|u(t)\|_{H^s}  \le C_s\bigl(\|u_0\|_{H^s}\bigr)=:M<\infty.$$
By the uniform local well-posedness theory, every datum $v_0\in H^s(\mathbb R)$ satisfying $\|v_0\|_{H^s}\le M+1$ generates a solution on a time interval of length $\delta>0$. Choose $t_0<T^*$ sufficiently close to $T^*$ so that
$$ T^*-t_0<\frac{\delta}{2}.$$
Since $\|u(t_0)\|_{H^s}\le M$, the local theory applied at time $t_0$ produces a solution on $ [t_0,t_0+\delta)$. By uniqueness, this solution agrees with $u$ on the overlap with $I_{\max}$. Since $ t_0+\delta>T^*$, this contradicts the maximality of $I_{\max}$. Consequently, $T^*=\infty$.

Finally, for $s\ge1$, global well-posedness follows from the result at any fixed regularity $s_0\in[3/4,1)$ together with the persistence of higher regularity. Consequently, \eqref{maineq} is globally well-posed in $H^s(\mathbb R)$ for all $s\ge3/4$.
\end{proof}

{\bf Acknowledgments.}

The work of M. Chen is partially supported by the National Natural Science Foundation of China (NSFC) grant 12001236, the Guangdong Basic and Applied Basic Research Foundation grant 2025A1515011925, and Science and Technology Projects in Guangzhou under grant 2025A04J3568. The work of Z. Wang is partially supported by the NSFC grant 11901092.

\end{document}